\documentclass[leqno,11pt, a4]{amsart}
\usepackage[utf8]{inputenc}
\usepackage{amsthm}
\usepackage{amsmath}
\usepackage{enumitem}
\usepackage{amsfonts}
\usepackage{amssymb}
\usepackage{graphicx}
\usepackage{pdfsync}
\usepackage{xcolor}
\newtheorem{theorem}{Theorem}[section]
\newtheorem{definition}{Definition} [section]
\newtheorem{corollary}{Corollary}[theorem]
\newtheorem{lemma}[theorem]{Lemma}
\newtheorem{proposition}[theorem]{Proposition}
\newtheorem{remark}[theorem]{Remark}
\newcounter{yuppo}
\newcommand{\Keler} {K\"{a}hler }
\newcommand{\keler} {K\"{a}hler}

\newcommand{\End}{\operatorname{End}}

\newcommand{\cds}{\cdots}
\newcommand{\cd}{\cdot}
\renewcommand{\setminus}{-}
\newcommand{\om}{\omega}

\renewcommand{\phi}{\varphi}
\newcommand{\cinf}{C^\infty}

\newcommand{\ra}{\rightarrow}
\newcommand{\lra}{\longrightarrow}
\newcommand{\C}{\mathbb{C}}
\newcommand{\R}{\mathbb{R}}

\newcommand{\Gl}{\operatorname{Gl}}

\newcommand{\restr}[1]          {\vert_{#1}}

\newcommand{\Ad}{\operatorname{Ad}}
\newcommand{\ad}{{\operatorname{ad}}}

\newcommand{\ga}{\gamma}

\newcommand{\meno}{^{-1}}

\newcommand{\enf}{\emph}

\newcommand{\liu}{\mathfrak{u}}

\newcommand{\lia}{\mathfrak{a}}

\newcommand{\liek}{\mathfrak{k}}
\newcommand{\lier}{\mathfrak{r}}

\newcommand{\lieg}{\mathfrak{g}}

\newcommand{\liep}{\mathfrak{p}}
\newcommand{\lieq}{\mathfrak{q}}
\newcommand{\liez}{\mathfrak{z}}

\newcommand{\liem}{\mathfrak{m}}
\newcommand{\lien}{\mathfrak{n}}

\newcommand{\la}{\lambda}

\newcommand{\alfa}{\alpha}
 \newcommand{\vacuo}{\emptyset}
\newcommand{\OO}{\mathcal{O}} 
\newcommand{\ext}{\operatorname{ext}} 
\newcommand{\sx}{\langle} 
\newcommand{\xs}{\rangle}
\newcommand{\relint}{\operatorname{relint}} 

\newcommand{\Crit}{\operatorname{Crit}} 
\newcommand{\roots}{\Delta} 
\newcommand{\simple}{\Pi} 

\newcommand{\spam}{\operatorname{span}}

\newcommand{\CF}{C_F}

\newcommand{\noparty}[1]{}
\newcommand{\changed}[1]{{#1}}

\newcommand{\scalo}{\sx \, , \, \xs}

\newcommand{\metrica}{(\, , \, )}

\newcommand{\mup}{\mu_\liep}
\newcommand{\mupb}{\mu_\liep^\beta}

\title{Properties of Gradient maps associated with Action of Real reductive Group}
\author{Biliotti, L.}
\address{Leonardo Biliotti, Dipartimento di Scienze Matematiche, Fisiche e Informatiche \\
          Universit\`a di Parma (Italy)}
\email{leonardo.biliotti@unipr.it}
\author{Oluwagbenga Joshua Windare}
\address{Oluwagbenga Joshua Windare, Dipartimento di Scienze Matematiche, Fisiche e Informatiche \\
          Universit\`a di Parma (Italy)}
\email{oluwagbengajoshua.windare@unipr.it}
\thanks{The first author was partially supported by the Project PRIN 2015, ``Real and Complex Manifolds: Geometry, Topology and Harmonic Analysis'',
Project PRIN  2017 ``Real and Complex Manifolds: Topology, Geometry and holomorphic dynamics'' and by GNSAGA INdAM.}

\begin{document}
\maketitle
\begin{abstract}
\noindent
Let $(Z,\omega)$ be a \Keler manifold and  let $U$ be a compact connected Lie group with Lie algebra $\mathfrak{u}$ acting on $Z$ and preserving $\omega$. We assume that the $U$-action extends holomorphically to an action of the complexified group $U^\C$ and the $U$-action on $Z$ is Hamiltonian. Then there exists a $U$-equivariant momentum map $\mu : Z\to \mathfrak{u}$.
If $G\subset U^\C$ is a closed subgroup such that the Cartan decomposition $U^\C = U\text{exp}(i\mathfrak{u})$ induces a Cartan decomposition
$G = K\text{exp}(\mathfrak{p}),$ where $K = U\cap G$, $\mathfrak{p} = \mathfrak{g}\cap i\mathfrak{u}$ and $\mathfrak{g}=\mathfrak k \oplus \mathfrak p$ is the Lie algebra of $G$, there is
a corresponding gradient map $\mu_\mathfrak{p} : Z\to \mathfrak{p}$. If $X$ is a $G$-invariant compact and connected real submanifold of $Z,$ we may consider
$\mu_{\mathfrak p}$ as a mapping $\mu_\mathfrak{p} : X\to \mathfrak{p}.$ Given an $\mathrm{Ad}(K)$-invariant scalar product on $\mathfrak p$, we obtain a Morse like function $f=\frac{1}{2}\parallel \mu_{\mathfrak p} \parallel^2$ on $X$. We point out that, without the assumption that $X$ is real analytic manifold, the Lojasiewicz gradient inequality holds for $f$. Therefore the limit of the negative gradient flow of $f$ exists and it is unique. Moreover, we prove that any $G$-orbit collapses to a single $K$-orbit and two critical points of $f$ which are in the same $G$-orbit belong to the same $K$-orbit.  We also investigate convexity properties of the gradient map $\mu_\mathfrak{p}$ in the Abelian  cases. In particular, we study two orbits variety $X$ and we investigate topological and cohomological properties of $X$.
\end{abstract}
\newpage
\renewcommand{\contentsname}{Table of Contents}
\tableofcontents
\section{Introduction}
\pagenumbering{arabic}
In this paper, we study the actions of real reductive Lie groups on real submanifolds of \Keler manifolds.

Let $U$ be a compact connected Lie group with Lie algebra $\mathfrak{u}$ and let $U^\C$ be its complexification. We say that a subgroup $G$ of $U^\C$ is
compatible if $G$ is closed and the map $K\times \mathfrak{p} \to G,$ $(k,\beta) \mapsto k\text{exp}(\beta)$ is a diffeomorpism where $K := G\cap U$
and $\mathfrak{p} := \mathfrak{g}\cap \text{i}\mathfrak{u};$ $\mathfrak{g}$ is the Lie algebra of $G.$ The Lie algebra $\mathfrak{u}^\C$ of $U^\C$ is the direct
sum $\mathfrak{u}\oplus i\mathfrak{u}.$ It follows that $G$ is compatible with the Cartan decomposition $U^\C = U\text{exp}(\text{i}\mathfrak{u})$
(see Section \ref{comp-subgrous}), $K$ is a maximal compact subgroup of $G$ with Lie algebra $\mathfrak{k}$ and that
$\mathfrak{g} = \mathfrak{k}\oplus \mathfrak{p}.$

Let $(Z,\omega)$ be a \Keler manifold with an holomorphic action of the complex reductive group $U^\C$. We also assume $\omega$ is $U$-invariant and that there is a $U$-equivariant moment map $\mu : Z \to \mathfrak{u}^*.$ By definition, for any $\xi \in \mathfrak{u}$ and $z\in Z,$ $d\mu^\xi = i_{\xi_Z}\omega,$ where $\mu^\xi(z) := \mu(z) (\xi)$, and $\xi_Z$ denotes the fundamental vector field induced on $Z$ by the action of $U,$
$$
\xi_Z(z) := \frac{d}{dt}\bigg \vert_{t=0} \text{exp}(t\xi)\cdot z.
$$

The inclusion i$\mathfrak{p}\hookrightarrow \mathfrak{u}$ induces by restriction, a $K$-equivariant map $\mu_{\text{i}\mathfrak{p}} : Z \to (\text{i}\mathfrak{p})^*.$ Using a $\mathrm{Ad}(U)$-invariant inner product on $\mathfrak u^\C$ to  identify $(\text{i}\mathfrak{p})^*$ and $\mathfrak{p},$ so $\mu_{\text{i}\mathfrak{p}}$ can be viewed as a map $\mu_{\mathfrak{p}} : Z \to \mathfrak{p}.$  For $\beta \in \mathfrak{p}$ let $\mu_\mathfrak{p}^\beta$ denote $\mu^{-\text{i}\beta}.$ i.e., $\mu_\mathfrak{p}^\beta(z) = -\sx\mu(z), i\beta\xs.$  Then grad$\mu_\mathfrak{p}^\beta = \beta_Z,$ where grad is computed with respect to the Riemannian metric induced by the \Keler structure. The map $\mu_\mathfrak{p}$ is called the gradient map associated with $\mu$ (see Section \ref{subsection-gradient-moment}). For a $G$-stable locally closed real submanifold $X$ of $Z,$ we consider $\mu_\mathfrak{p}$ as a mapping $\mu_\mathfrak{p} : X\to \mathfrak{p}.$ Using the inner product  on $\mathfrak p\subset i\mathfrak u$, we define the norm square of $\mu_\mathfrak{p}$ by
$$f(x) := \frac{1}{2}\parallel \mu_\mathfrak{p}(x) \parallel^2; \quad x \in X.$$
The set of semistable points associated with the critical points of the norm square of $\mu_\mathfrak{p}$ was studied in great details in \cite{heinzner-schwarz-stoetzel}. The norm square of $\mu_\mathfrak{p}$ is in general far from being Morse-Bott and it's critical sets may be very complicated. But can we have a particular case when the norm square will be Morse-Bott? We find this question to be true for two orbits variety. From now on, we always assume that $X$ is connected and compact.

Suppose that the action of $G$ on $X$ has two orbits. $X$ is called a two orbit variety. S. Cupit-Foutou obtained the classification of a complex algebraic varieties on which a reductive complex algebraic group acts with two orbits \cite{Cupit-Foutou}. Applying standard
Morse-theoretic results in \cite{Kirwan} and \cite{heinzner-schwarz-stoetzel}, we prove that the norm square is Morse-Bott and obtain information on the
cohomology and $K$-equivariant cohomology of $X$ (Theorem \ref{two-orbit-Morse}), generalizing \cite{Anna}.

A central ingredient to prove this result is the Ness Uniqueness Theorem which asserts that any two critical points of $f$ in the same $G$-orbit in fact belong to the same $K$-orbit (Theorem \ref{critt}). Moreover,  although we do not assume that $X$ is real analytic manifold, we point out that for any $G$-invariant compact and connected submanifold of $Z$, the Lojasiewicz gradient inequality holds for the norm square. Therefore, the limit of the negative gradient flow exists and it is unique and any $G$-orbit collapses to a single $K$-orbit (Theorem \ref{corr}).
 We use the original ideas from \cite{Salamon} in a different context. By the stratification theorem, we have
$$
\{p\in X : \overline{G \cdot p}\cap \mu_\mathfrak{p}^{-1}(0) \neq \emptyset \} = \{p\in X : \lim_{t\to +\infty}\phi_t(p)\in \mu_\mathfrak{p}^{-1}(0)\} = S_G(\mu_\mathfrak{p}^{-1}(0)),
$$
where $\phi_t(p)$ is the flow of the vector field -grad$f.$ Then, there exist a $K$-equivariant strong deformation of $S_G(\mu_\mathfrak{p}^{-1}(0))$ onto the set $\mu_\mathfrak{p}^{-1}(0)$ (Theorem \ref{retraction-theorem}). Hence no analyticity assumption is necessary in the statement of the retraction Theorem answering Question $1$ in \cite[$p.219$]{heinzner-stoetzel}.

Biliotti and Ghigi \cite{LG} proved a convexity theorem along orbits in a very general setting using only so-called Kempf-Ness function. The behaviour of the corresponding gradient map is encoded in the Kempf-Ness function. Recently, Biliotti \cite{LB} gives a new proof of the Hilbert-Mumford criterion for real reductive Lie groups stressing the properties of the Kempf-Ness functions. He shows that the Kempf-Ness function is Morse-Bott and it is convex along geodesics for the action of a linear group on $\mathbb{P}(V)$ where $V$ is a finite dimensional dimensional vector space. We prove this result in a general setting.

Results on convexity theorems are obtained. Let $\mathfrak{a}\subset\mathfrak{p}$ be an Abelian subalgebra. Let $\mu_\mathfrak{a} : X\to \mathfrak{a}$ denote the gradient map of $A = \text{exp}(\mathfrak{a}).$ If $\pi_\mathfrak{a} : \mathfrak{p} \to \mathfrak{a}$ is the orthogonal projection, then $\mu_\mathfrak{a} = \pi_\mathfrak{a} \circ \mu_\mathfrak{p}.$ Although there is no counterexample, we do not know if the Abelian Convexity Theorem holds for any $G$-invariant connected  submanifold (see for instance \cite{biliotti-ghigi-heinzner-document,LG,heinzner-schuetzdeller} for more details on the subject).
If $G$ has a unique closed orbit $\OO,$ then we prove that $\mu_\mathfrak{a}(X) = \mu_\mathfrak{a}(\OO)$ and so a polytope. This result is new also if $G = U^\C$ and $X = Z.$ This means that $\OO$ captures all the information of the $A$-gradient map. As an application, we prove that the Abelian convexity Theorem holds for a two orbits variety.

If $Z$ is connected and compact and $X\subset Z$ is a $A$-stable compact, connected coisotropic submanifold of $Z$, then we prove $\mu_\mathfrak{a}(X) = \mu_\mathfrak{a}(Z)$ (Theorem \ref{convexity-coisotropic}) and so it is a polytope as well. More precisely, there exists an open and dense subset $W$ of $X$  such that for any $p\in W,$ we have
$\mu_\mathfrak{a}(X) = \mu_\mathfrak{a}(\overline{A\cdot p})$. 

\section{Preliminaries}

 \subsection{Convex geometry}
\label{pre-convex}

 In this section, some definitions and results in convex geometry are recalled. The reader can see e.g. \cite{schneider-convex-bodies} and \cite{biliotti-ghigi-heinzner-1-preprint} for further details on the topic. 
 
 Let $V$ be a real vector space with a scalar product $\scalo$ and let $E\subset V$ be a \changed{compact} convex subset. The \emph{relative interior} of $E$, denoted by $\relint E$, is the interior of $E$ in its affine hull. For $a,b\in E,$ denote the closed segment joining $a$ and $b$ by $[a,b]$. Then, a face of $E$ is a convex subset $F$ of $E$ such that if $a,b\in E$ and $\relint[a,b]\cap F\neq \vacuo$, then
 $[a,b]\subset F$.  The \emph{extreme points} of $E$ denoted by $\ext E$ are the points
 $a\in E$ such that $\{a\}$ is a face. By a Theorem of Minkowski, $E$ is the convex hull of its extreme points \cite[p.19]{schneider-convex-bodies}. The faces of $E$
 are closed \cite[p. 62]{schneider-convex-bodies}. The empty set and $E$ are faces of $E:$ the other faces are called \enf{proper}.
 \begin{definition}
  The support function of $E$ is defined by the function $ h_E : V \ra \R$, $
 h_E(u) = \max \{\sx x, u \xs : x\in E\}$.  If $ u\neq 0$, the
 hyperplane $H(E, u) : = \{ x\in E : \sx x, u \xs = h_E(u)\}$ is
 called the supporting hyperplane of $E$ for $u$.

 The set
   \begin{gather}
     \label{def-exposed}
     F_u (E) : = E \cap H(E,u)
   \end{gather}
   is a face and it is called the \enf{exposed face} of $E$ defined by
   $u$.
 \end{definition}

 Intuitively, the meaning of the support function is simple. For instance, consider a nonempty closed convex set $E\subset \mathbb{R}^n.$ Then for a unit vector $u\in S^{n-1}\cap \text{dom}h_E,$ the supporting function $h_E(u)$ is the signed distance of the support plane to $E$ with exterior normal vector $u$ from the origin; the distance is negative if and only if $u$ points into the open half-space containing the origin.  In general not all faces of a convex subset are exposed. For instance, consider the convex hull of a closed disc and a point outside the disc: the resulting convex set is the union of the
 disc and a triangle. The two vertices of the triangle that lie on the
 boundary of the disc are non-exposed 0-faces.

 A subset $E\subset V$ is called a \emph{convex cone} if $E$ is convex, non empty and closed under multiplication by non negative real numbers.
The following results about a compact convex set $E$ and it's faces are recalled from \cite{biliotti-ghigi-heinzner-1-preprint}
 \begin{lemma}
   \label{u-cono}
   If $F \subset E$ is an exposed face, the set $\CF : = \{ u\in V:
   F=F_u(E) \}$ is a convex cone. If $G$ is a compact subgroup of
   $O(V)$ that preserves both $E$ and $F$, then $\CF$ contains a fixed
   point of $G$.
 \end{lemma}
\begin{theorem} [\protect{\cite[p. 62]{schneider-convex-bodies}}]
   \label{schneider-facce} If $E$ is a compact convex set and
   $F_1,F_2$ are distinct faces of $E$, then $\relint F_1 \cap \relint
   F_2=\vacuo$. If $G$ is a nonempty convex subset of $ E$ which is
   open in its affine hull, then $G \subset\relint F$ for some face
   $F$ of $E$. Therefore $E$ is the disjoint union of \changed{the
     relative interiors of its} faces.
 \end{theorem}
The next result is a possibly well known but useful fact.
\begin{proposition}\label{convex-criterium}
Let $C_1 \subseteq C_2$ be two compact convex set of $V$. Assume that for any $\beta \in V$ we have
\[
\mathrm{max}_{y\in C_1} \langle y , \beta \rangle=\mathrm{max}_{y\in C_2} \langle y , \beta \rangle.
\]
Then $C_1=C_2$.
\end{proposition}
\begin{proof}
We may assume without loss of generality  that the affine hull of $C_2$ is $V$.
Assume by contradiction that $C_1 \subsetneq C_2$. Since $C_1$ and $C_2$ are both compact, it follows that there exists $p\in \partial C_1$ such that $p\in \stackrel{o}{C_2}$. Since every face of a convex compact set is contained in an exposed face \cite{schneider-convex-bodies}, there exists $\beta \in V$ such that
\[
\mathrm{max}_{y\in C_1} \langle y , \beta \rangle=\langle p, \beta \rangle.
\]
This means the linear function $x\mapsto \langle x, \beta \rangle$ restricted on $C_2$ achieves its maximum at an interior point which is a contradiction.
\end{proof}

 \subsection{Compatible subgroups}
\label{comp-subgrous}
In this section, the notion of compatible subgroup is discussed. 
Let $H$ be a Lie group with Lie algebra $\mathfrak{h}$ and $E, F \subset \mathfrak{h}$. Then, we set
 \begin{gather*}
   E^F :=     \{\eta\in E: [\eta, \xi ] =0, \forall \xi \in F\} \\
   H^F = \{g\in H: \Ad (g) (\xi ) = \xi, \forall \xi \in F\}.
 \end{gather*}
 If $F=\{\beta\}$ we write simply $E^\beta$ and $H^\beta$.
 
 Let $U$ be a
 compact Lie group and let $U^\C$ be its universal complexification in the sense of \cite{ho}.  The group $U^\C$
 is a reductive complex algebraic group with lie algebra $\mathfrak{u}^\C$ \cite{chevalley}. $\mathfrak{u}^\C$ have the Cartan decomposition
 $$
 \mathfrak{u}^\C = \mathfrak{u} + i\mathfrak{u}
 $$ with a conjugation map $\theta : \liu^\C \ra \liu^\C$ and the corresponding group isomorphism $\theta : U^\C \ra
   U^\C$.  Let $f: U \times i\liu \ra U^\C$ be the diffeomorphism
 $f(g, \xi) = g \text{exp}  (\xi)$. The decomposition $U^\C = U\exp(i\mathfrak{u})$ is referred to as the Cartan decomposition.

 Let $G\subset U^\C$ be a closed real
 subgroup of $U^\C$ We say
 that $G$ is \enf{compatible} with the Cartan decomposition of $U^\C$ if $f (K \times \liep) = G$ where $K:=G\cap U$ and $\liep:= \lieg \cap i\liu$. The
 restriction of $f$ to $K\times \liep$ is then a diffeomorphism onto
 $G$. It follows that $K$ is a maximal compact subgroup of $G$ and
 that $\lieg = \liek \oplus \liep$.  Note that $G$ has finitely many
 connected components. Since $U$ can be embedded in $\Gl(N,\C)$ for
 some $N$, and any such embedding induces a closed embedding of
 $U^\C$, any compatible subgroup is a closed linear group. By \cite[Proposition 1.59
p.57]{knapp-beyond}, $\mathfrak g$ is a real reductive Lie algebra, and so $\mathfrak g=\mathfrak{z(g)}\oplus z[\mathfrak g, \mathfrak g]$, where $\mathfrak{z(g)} := \{v\in \mathfrak g:\, [v,\mathfrak g]=0\}$
is the Lie algebra of the center of G. Denote by $G_{ss}$ the analytic
 subgroup tangent to $[\lieg, \lieg]$ and by $Z(G)$ the center of $G$. Then $G_{ss}$ is closed and
 $G^o=Z(G)^o\cd G_{ss}$ \cite[p. 442]{knapp-beyond}, where $G^o$, respectively $Z(G)^o$, denotes the connected component of the identity of $G$, respectively of $Z(G)$.
\begin{lemma}[\protect{\cite[Lemma 7]{LA}}]
\label{lemcomp}$\,$
   \begin{enumerate}
   \item \label {lemcomp1} If $G\subset U^\C$ is a compatible
     subgroup, and $H\subset G$ is closed and $\theta$-invariant,
     then $H$ is compatible if and only if $H$ has only finitely many connected components.
   \item \label {lemcomp2} If $G\subset U^\C$ is a connected
     compatible subgroup, then $G_{ss}$ is compatible.
     \item \label{lemcomp3} If $G\subset U^\C$ is a compatible
       subgroup, and $E\subset \liep$ is any subset, then $G^E$ is
       compatible. Indeed, $G^E=K^E\cap \exp(\liep^E)$, where $K^E=K\cap G^E$ and $\liep^E=\{v\in \liep:\, [v,E]=0\}$.
   \end{enumerate}
 \end{lemma}

\subsection{Parabolic subgroups}
Let $G \subset U^\C$ be a compatible, and let $\lieg = \liek \oplus \liep$ be the corresponding decomposition.  A subalgebra $\lieq \subset \lieg$ is \enf{parabolic} if
$\lieq^\C$ is a parabolic subalgebra of $\lieg^\C$.  One way to
describe the parabolic subalgebras of $\lieg$ is by means of
restricted roots.  If $\lia \subset \liep$ is a maximal subalgebra,
let $\roots(\lieg, \lia)$ be the (restricted) roots of $\lieg$ with
respect to $\lia$, let $\lieg_\la$ denote the root space corresponding
to $\la$ and let $\lieg_0 = \liem \oplus \lia$, where $\liem =
\liez_\liek(\lia)$.  Let $\simple \subset \roots(\lieg, \lia)$ be a
base and let $\roots_+$ be the set of positive roots. If $I\subset
\simple$ set $\roots_I : = \spam(I) \cap \roots$. Then
\begin{gather}
  \label{para-dec}
  \lieq_I:= \lieg_0 \oplus \bigoplus_{\la \in \roots_I \cup \roots_+}
  \lieg_\la
\end{gather}
is a parabolic subalgebra. Conversely, if $\lieq \subset \lieg$ is a
parabolic subalgebra, then there are a maximal subalgebra $\lia
\subset \liep$ contained in $\lieq$, a base $\simple \subset
\roots(\lieg, \lia)$ and a subset $I\subset \simple $ such that $\lieq
= \lieq_I$.  We can further introduce
\begin{gather}
\label{notaz-I}
\begin{gathered}
  \lia_I : = \bigcap_{\la \in I} \ker \la \qquad \lia^I := \lia_I^\perp \\
  \lien_I = \bigoplus_{\la \in \roots_+ \setminus \roots_I} \lieg_\la
  \qquad \liem_I : = \liem \oplus \lia^I \oplus \bigoplus_{\la \in
    \roots_I}\lieg_\la.
\end{gathered}
\end{gather}
Then $\lieq_I = \liem_I \oplus \lia_I \oplus \lien_I$. Since
$\theta\lieg _ \la = \lieg_{-\la}$, it follows that $ \lieq_I \cap
\theta\lieq_I = \lia_I \oplus \liem_I$.  This latter \changed{Lie algebra} coincides
with the centralizer of $\lia_I$ in $\lieg$. It is a \enf{Levi factor}
of $\lieq_I$ and
\begin{gather}
  \label{liaI}
  \lia_I =\liez (\lieq_I \cap \theta\lieq_I) \cap \liep.
\end{gather}
Another way to describe parabolic subalgebras of $\lieg$ is the
following.  If $\beta \in \liep$, the endomorphism $\ad (\beta) \in \End
\lieg$ is diagonalizable over $\R$. Denote by $ V_\la (\ad (\beta)) $ the
eigenspace of $\ad (\beta)$ corresponding to the eigenvalue $\la$.  Set
\begin{gather*}
  \lieg^{\beta+}: = \bigoplus_{\la \geq 0} V_\la (\ad (\beta)).
\end{gather*}
\begin{lemma}
  \label{para-beta}\cite[Lemma 8]{LA}
  For any $\beta$ in $ \liep$, $\lieg^{\beta+}$ is a parabolic
  subalgebra of $\lieg$.  If $\lieq \subset \lieg$ is a parabolic
  subalgebra, there is some vector $\beta \in \liep$ such that $\lieq
  = \lieg^{\beta+}$.  The set of all such vectors is an open convex
  cone in
$\liez(\lieq \cap \theta\lieq) \cap \liep$.
\end{lemma}
A \enf {parabolic subgroup} of $G$ is a subgroup of the form $Q =
N_G(\lieq)$ where $\lieq $ is a parabolic subalgebra of $\lieg$.
Equivalently, a parabolic subgroup of $G$ is a subgroup of the form
$P\cap G$ where $P$ is parabolic subgroup of $G^\C$ and $\liep$ is the
complexification of a subspace $\lieq \subset \lieg$.  If $\beta
\in \liep$ set
\begin{gather*}
\begin{gathered}
  G^{\beta+} :=\{g \in G : \lim_{t\to - \infty} \text{exp} ({t\beta}) g
  \text{exp} ({-t\beta}) \text { exists} \}\\
  R^{\beta+} :=\{g \in G : \lim_{t\to - \infty} \text{exp} ({t\beta})
  g \text{exp} ({-t\beta}) =e \}\\
  G^\beta=\{g\in G:\, \mathrm{Ad}(g)(\beta)=\beta\}
\end{gathered}
\qquad
  \lier^{\beta+}: = \bigoplus_{\la > 0} V_\la (\ad (\beta)).
\end{gather*}
Note that $ \lieg^{\beta+}= \lieg^\beta \oplus \lier^{\beta+}$.
\begin{lemma}\label{parabolicc}\cite[Lemma 9]{LA}
If $G$ is connected, then $G^{\beta +} $ is a parabolic subgroup of $G$ with Lie algebra
  $\lieg^{\beta+}$.  Every parabolic subgroup of $G$ equals
  $G^{\beta+}$ for some $\beta \in \liep$.  $R^{\beta+}$ is connected and it is the unipotent radical of $G^{\beta+}$ and $G^\beta$ is a Levi factor.
\end{lemma}
\subsection{Gradient map}
\label{subsection-gradient-moment}
Let $(Z, \om)$ be a \Keler manifold. Let $U^\C$ acts
holomorphically on $Z$ i.e., the action $U^\C \times Z \to Z$ is holomorphic. Assume that $U$ preserves $\om$ and that there is a $U$-equivariant
momentum map $\mu: Z \ra \liu$.  If $\xi \in \liu$, we denote by $\xi_Z$
the induced vector field on $Z$ and we let $\mu^\xi \in \cinf(Z)$ be
the function $\mu^\xi(z) := \sx \mu(z),\xi\xs$, where $\sx \cdot , \cdot \xs$ is an $\mathrm{Ad}(U)$-invariant scalar product on $\mathfrak u^\C$. We may also assume that the  multiplication by $i$ is an isometry from $\mathfrak u$ onto $i\mathfrak u$ and $\langle \liu , i \liu \rangle =0$ (\cite[$p. 428$]{LAP}. By definition, we have
$$
d\mu^\xi =
i_{\xi_Z} \om.
$$

Let $G \subset U^\C$ be a compatible subgroup of $U^\C$.
For $x\in Z$, let $\mu_\mathfrak{p}^\beta(x)$ denote $-i$ times the component of $\mu (x)$ along $\beta$ in the direction of $i\mathfrak{p},$ i.e.,
\begin{equation}\label{mu3}
\mu_\mathfrak{p}^\beta(x) := \langle \mu_\mathfrak{p}(x), \beta \rangle =\langle i \mu(x),\beta \rangle = \langle \mu(x), -i\beta \rangle = \mu^{-i\beta} (x)
\end{equation} for any $\beta \in \mathfrak{p}.$
Then, the map defined by
\begin{gather*}
  \mu_\liep : Z \ra \liep
\end{gather*} is called the $G$ \emph{gradient map}.
Let $\metrica$ be the \Keler
metric associated to $\om$, i.e. $(v, w) = \om (v, Jw),$ for all $z\in Z$ and $v,w\in T_zZ$ where $J$ denotes the complex structure on $TZ$. Then
$\beta_Z $ is the gradient of $\mup^\beta$, where $\beta_Z$ is the vector field on $Z$ corresponding to $\beta$ and the gradient is computed with respect to $(\cdot,\cdot)$. For the rest of this paper, fix a $G$-invariant locally closed submanifold $X$ of $Z.$ We denote the restriction of $\mu_\mathfrak{p}$ to $X$ by $\mu_\mathfrak{p}.$ Then
$$
\text{grad}\mu_\mathfrak{p}^\beta = \beta_X,
$$ where grad is now computed with respect to the induced Riemannian metric on $X.$ Since $X$ is $G$-stable, $\beta_X=\beta_Z$. Similarly, $\bot$ denotes perpendicularity relative to the Riemannian metric on $X.$
We will now recall some of the properties of the gradient map.
\begin{lemma}
Let $x\in X$ and let $\beta \in \mathfrak{p}.$ Then either $\beta_X(x) = 0$ or the function $t\mapsto \mu_\mathfrak{p}^\beta(\text{exp}(t\beta)\cdot x)$ is strictly increasing.
\end{lemma}

\begin{proof}
Let $f(t) = \mu_\mathfrak{p}^\beta(\text{exp}(t\beta)\cdot x) = \langle \mu_\mathfrak{p}(\text{exp}(t\beta)\cdot x), \beta \rangle.$ Then
$$f'(t) = ( \beta_X(\text{exp}(t\beta)\cdot x), \beta_X(\text{exp}(t\beta)\cdot x)  \geq 0.
$$
Therefore $\beta_X (x)=0$ or the function $f$ is strictly increasing. 
\end{proof}
For any subspace $\mathfrak{m}$ of $\mathfrak{g}$ and $x\in X,$ let $$\mathfrak{m}\cdot x := \{\xi_X(x) : \xi \in \mathfrak{m}\}.$$
\begin{lemma}
Let $x\in X$. Then $$\text{ker}\,\, d\mu_\mathfrak{p}(x) = (\mathfrak{p}\cdot x)^\bot $$
\end{lemma}

\begin{proof}
From (\ref{mu3}),  $v\in \text{ker}\,\, d\mu_\mathfrak{p}(x)$ if and only if for all $\beta \in \mathfrak{p}$
\begin{align*}
     & \langle d\mu_\mathfrak{p}(x)(v), \beta \rangle = 0\\
     & \Longleftrightarrow d\mu_\mathfrak{p}^\beta (v) = 0\\
    & \Longleftrightarrow \langle \beta_X(x), v\rangle = 0.
\end{align*}
\end{proof}

\begin{lemma}
Let $x\in X.$ The following are equivalent:

\begin{enumerate}
    \item $d\mu_\mathfrak{p} : T_xX \to \mathfrak{p}$ is onto;
    \item $d\mu_\mathfrak{p} : \mathfrak{g}\cdot x \to \mathfrak{p}$ is onto;
    \item the map $\mathfrak p  \to T_x X$, $\beta \mapsto \beta_X$, is injective.
\end{enumerate}
\end{lemma}
\begin{proof}
$$\text{ker}\,\, d\mu_\mathfrak{p}(x) = (\mathfrak{p}\cdot x)^\bot,$$ it follows that $d\mu_\mathfrak{p}(x)$ \text{is surjective} if and only if
$d\mu_\mathfrak{p} : \mathfrak{p}\cdot x \to \mathfrak{p}$ \text{is surjective} if and only if $\dim \mathfrak p \cdot x=\dim \mathfrak p$,  concluding the proof.
\end{proof}
We recall the Slice Theorem, see \cite{heinzner-schwarz-stoetzel}. For any Lie group $G,$ a closed subgroup $H$ and any set $S$ with an $H-$action, the $G-$bundle over $G/H$ associated with the $H-$principal bundle $G\to G/H$ is denoted by $G\times^H S.$ This is the orbit space of the $H-$action on $G\times S$ given by $h\cdot(g,s) = (gh^{-1}, h\cdot s)$ where $g\in G,$ $s\in S$ and $h\in H.$ The $H-$orbit of $(g,s),$ considered as a point in $G\times^H S,$ is denoted by $[g,s].$
\begin{theorem}
  [Slice Theorem \protect{\cite[Thm. 3.1]{heinzner-schwarz-stoetzel}}]
  If $x \in X$ and $\mup(x) = 0$, there are a $G_x$-invariant
  decomposition $T_xX = \lieg \cd x \oplus W$, open $G_x$-invariant
  subsets $S \subset W$, $\Omega \subset X$ and a $G$-equivariant
  diffeomorphism $\Psi : G \times^{G_x}S \ra \Omega$, such that $0\in
  S, x\in \Omega$ and $\Psi ([e, 0]) =x$.
\end{theorem}
Let $\beta \in \mathfrak{p}$. It is well-known that $G^\beta$ is compatible and
\[
G^\beta=K^\beta \text{exp}  (\liep^\beta),
\]
where $K^\beta=K\cap G^\beta=\{g\in K:\, \mathrm{Ad}(g)(\beta)=\beta\}$ and
$\liep^\beta=\{v\in \liep:\, [v,\beta]=0\}$ (see \cite{knapp-beyond}).

\begin{corollary} \label{slice-cor} If $x \in X$ and $\mup(x) = \beta$,
  there are a $G^\beta$-invariant decomposition $T_xX = \lieg^\beta
  \cd x \, \oplus W$, open $G^\beta$-invariant subsets $S \subset W$,
  $\Omega \subset X$ and a $G^\beta$-equivariant diffeomorphism $\Psi
  : G^\beta \times^{G_x}S \ra \Omega$, such that $0\in S, x\in \Omega$
  and $\Psi ([e, 0]) =x$.
\end{corollary}
This follows applying the previous theorem to the action of $G^\beta$ on $X$. Indeed, by Lemma \ref{lemcomp} $G^\beta=K^\beta\exp(\liep^\beta)$ is compatible and the orthogonal projection of $\textbf{i} \mu$ onto $\liep^\beta$ is the $G^\beta$-gradient map $\mu_{\liep^\beta}$. The group $G^\beta$ is also compatible with the Cartan decomposition of $(U^\C )^{\beta}=(U^\C )^{\textbf{i} \beta}=(U^{\textbf{i}\beta} )^\C$ and $\textbf{i} \beta$ is fixed by the $U^{\textbf{i}\beta}$-action on $\liu^{\textbf{i}\beta}$. This implies that $\widehat{\mu_{\liu^{\textbf{i} \beta}}}:Z \lra \liu^{\textbf{i} \beta}$
given by $\widehat{\mu_{\liu^{\textbf{i}\beta}}(z)}=\pi_{\liu^{\textbf{i}\beta}} \circ \mu+\textbf{i}\beta$, where $\pi_{\liu^{\textbf{i}\beta}}$ is the orthogonal projection of $\liu$ onto  $\liu^{\textbf{i}\beta}$, is the $U^{\textbf{i}\beta}$-shifted momentum map. The associated $G^{\beta}$-gradient map is given by $\widehat{\mu_{\liep^\beta}} := \mu_{\liep^\beta} -
\beta$. Hence, if $G$ is commutative, then we have a Slice Theorem for $G$ at every point of $X$,
see \cite[p.$169$]{heinzner-schwarz-stoetzel} for more details.


If $\beta \in \liep$, \changed{then $\beta_X$ is a vector field on
  $X$, i.e. a section of $TX$. For $x\in X$, the differential is a map
  $T_xX \ra T_{\beta_X(x)}(TX)$. If $\beta_X(x) =0$, there is a
  canonical splitting $T_{\beta_X(x)}(TX) = T_xX \oplus
  T_xX$. Accordingly $d\beta_X(x)$ splits} into a horizontal and a
vertical part. The horizontal part is the identity map. We denote the
vertical part by $d\beta_X(x)$.  It belongs to $\End(T_xX)$.  Let
$\{\phi_t=\text{exp} (t\beta)\} $ be the flow of $\beta_X$.  \changed{There
  is a corresponding flow on $TX$. Since $\phi_t(x)=x$, the flow on
  $TX$ preserves $T_xX$ and there it is given by $d\phi_t(x) \in
  \Gl(T_xX)$.  Thus we get a linear $\R$-action on $T_xX$ with
  infinitesimal generator $d\beta_X(x) $.}
\begin{corollary}
  \label{slice-cor-2}
  If $\beta \in \liep $ and $x \in X$ is a critical point of $\mupb$,
  then there are open invariant neighbourhoods $S \subset T_xX$ and
  $\Omega \subset X$ and an $\R$-equivariant diffeomorphism $\Psi : S
  \ra \Omega$, such that $0\in S, x\in \Omega$, $\Psi ( 0) =x$. \changed{(Here
  $t\in \R$ acts as $d\phi_t(x)$ on $S$ and as $\phi_t$ on $\Omega$.)}
\end{corollary}
\begin{proof}
Since $\exp:\liep \lra G$ is a diffeomorphism onto the image, the subgroup $H:=\text{exp} (\R \beta)$ is compatible.  Hence, It is enough to
 apply the previous corollary to the $H$-action at $x$.
\end{proof}

Assume now that $\beta \in \liep$ and that $x \in
\Crit(\mu^\beta_\liep)$. Let $D^2\mup^\beta(x) $ denote the Hessian,
which is a symmetric operator on $T_xX$ such that
\begin{gather*}
  ( D^2 \mup^\beta(x) v, v) = \frac{\mathrm d^2}{\mathrm
    dt^2} 
  (\mup^\beta\circ \ga)(0)
\end{gather*}
where $\ga$ is a smooth curve, $\ga(0) = x$ and $ \dot{\ga}(0)=v$.
Denote by $V_-$ (respectively $V_+$) the sum of the eigenspaces of the
Hessian of $\mupb$ corresponding to negative (resp. positive)
eigenvalues. Denote by $V_0$ the kernel.  Since the Hessian is
symmetric we get an orthogonal decomposition
\begin{gather}
  \label{Dec-tangente}
  T_xX = V_- \oplus V_0 \oplus V_+.
\end{gather}
Let $\alfa : G \ra X$ be the orbit map: $\alfa(g) :=gx$.  The
differential $d\alfa_e$ is the map $\xi \mapsto \xi_X(x)$.
\begin{proposition}
    \label{tangent}
  If $\beta \in \liep$ and $x \in \Crit(\mu^\beta_\liep)$ then
  \begin{gather*}
    D^2\mup^\beta(x) = d \beta_X(x).
  \end{gather*}
  Moreover $d\alfa_e (\lier^{\beta\pm} ) \subset V_\pm$ and $d\alfa_e(
  \lieg^\beta) \subset V_0$.  If $X$ is $G$-homogeneous these are
  equalities.
\end{proposition}

\begin{proof}
  The first statement is proved in
  \cite[Prop. 2.5]{heinzner-schwarz-stoetzel}.  Denote by $\rho : G_x
  \ra T_xX$ the isotropy representation: $\rho(g) = dg_x$.  Observe
  that $\alfa$ is $G_x$-equivariant where $G_x$ acts on $G$ by
  conjugation, hence $d\alfa_e$ is $G_x$-equivariant, where $G_x$ acts
  on $\lieg$ by the adjoint representation and on $T_xX$ by the
  isotropy representation.  Since $\beta_X(x)=0$, $\text{exp} ({t\beta }) \in
  G_x$ for any $t$ and $d\alfa_e$ is $\R$-equivariant. Therefore it
  interchanges the infinitesimal generators of the $\R$-actions,
  i.e. $d\alfa_e \circ \ad \beta = d\beta_X = D^2\mupb(x)$.  The
  required inclusions follow. If $G$ acts transitively on $X$ we must
  have $T_xX = d\alfa_e(\lieg)$. Hence the three inclusions must be
  equalities.
\end{proof}
\begin{corollary}
  \label{MorseBott}
  For every $\beta \in \liep$, $\mupb$ is a Morse-Bott function.
\end{corollary}
\begin{proof}
  Let $X^\beta:=\{ x\in X: \beta_X(x) =0\}$.  Corollary
  \ref{slice-cor-2} implies that $X^\beta$ is \changed{a smooth
    submanifold}. Since $T_xX^\beta = V_0$ for $x\in X^\beta$, the
  first statement of Proposition \ref{tangent} shows that the Hessian
  is nondegenerate in the normal directions.
\end{proof}
From now on, we assume that $X$ is compact and connected.
Let $\mup:X \lra \liep$ the $G$ gradient map. Let $\beta \in \liep$.
Let $c_1 > \cds > c_r$ be the critical
  values of $\mup^\beta$. The corresponding level sets of $\mup^\beta$, $C_i:=(\mup^\beta)\meno ( c_i)$ are submanifolds which are union of
  components of $\Crit(\mup^\beta)$. The function $\mup^\beta$ defines a gradient flow generated by its gradient which is given by $\beta_X$.
By Corollary \ref{slice-cor-2}, it follows that for any $x\in X$ the limit:
  \begin{gather*}
    \phi_\infty (x) : = \lim_{t\to +\infty }
    \exp(t\beta ) x,
  \end{gather*}
  exists.
Let us denote by $W_i^\beta$ the \emph{unstable manifold} of the critical component $C_i$
  for the gradient flow of $\mup^\beta$:
  \begin{gather}
    \label{def-wu}
    W_i^\beta := \{ x\in X: \phi_\infty (x) \in C_i \}.
  \end{gather}
Applying Corollary \ref{MorseBott}, we have the following well-known decomposition of $X$ into unstable manifolds with respect to $\mup^\beta$.
\begin{theorem}\label{decomposition}
In the above assumption, we have
\begin{gather}
    \label{scompstabile}
    X = \bigsqcup_{i=1}^r W_i^\beta,
  \end{gather}
and for any $i$ the map:
  \begin{gather*}
    (\phi_\infty)\restr{W_i} : W_i^\beta \ra C_i,
  \end{gather*}
  is a smooth fibration with fibres diffeomorphic to $\R^{l_i}$ where
  ${l_i}$ is the index (of negativity) of the critical submanifold
  $C_i$
\end{theorem}
\section{The Norm Square of the Gradient Map}
\label{Norm_squared}
We assume throughout that $X$ is compact and connected $G$ invariant submanifold of $(Z,\omega)$ and $G$ is connected. Let $\mup:X \lra \liep$ denote the $G$ gradient map.
Let $\parallel \cdot\parallel$ denote the norm functions associated to the $\mathrm{Ad}(K)$-invariant scalar product $\langle\cdot , \cdot\rangle$ on $\mathfrak p$.
Define the function $f : X \rightarrow \mathbb{R}$ by
\begin{equation}\label{ns}
f(x) := \frac{1}{2}\parallel\mu_\mathfrak{p}(x)\parallel^2, \qquad \text{for}\quad x \in X.
\end{equation}
In this section, the critical points of this function will be of central importance.

\begin{lemma}\label{nmg}
The gradient of $f$ is given by
\begin{equation}\label{nmge}
    \triangledown f(x) = \beta_X(x), \quad \beta := \mu_\mathfrak{p}(x)\in \mathfrak{p} \quad \text{and}\quad x\in X.
\end{equation} Hence, $x\in X$ is a critical point of $f$ if and only if $\beta_X(x) = 0.$
\end{lemma}

\begin{proof}

Define a curve $\gamma(t)$ such that $\gamma(0) = x$ and $\gamma'(0) = v \in T_xX.$
$f(x) = \frac{1}{2}\parallel \mu_\mathfrak{p}(x)\parallel^2 = \frac{1}{2}\langle \mu_\mathfrak{p}(x), \mu_\mathfrak{p}(x)\rangle$
\begin{align*}
    df(x)v &= \frac{d}{dt}\bigg \vert_{t=0}f(\gamma(t))\\
    & = \frac{1}{2}\frac{d}{dt} \bigg \vert_{t=0}\langle \mu_\mathfrak{p}(\gamma(t)), \mu_\mathfrak{p}(\gamma(t))\rangle \\
    & = \langle d\mu_\mathfrak{p}(\gamma(t))\gamma'(t), \mu_\mathfrak{p}(\gamma(t))\rangle|_{t=0}\\
    & = \langle d\mu_\mathfrak{p}(x)v, \mu_\mathfrak{p}(x)\rangle = \langle \beta_X(x), v), \quad \beta = \mu_\mathfrak{p}(x).
\end{align*}
Hence, $\triangledown f(x) = \beta_X(x).$
\end{proof}

\begin{corollary}
  Let $x\in X$ and set $\beta := \mu_\mathfrak{p}(x).$ The following are equivalent.
 \begin{enumerate}
      \item $\beta_X(x) = 0,$
      \item $d\mu_\mathfrak{p}^\xi(x) = 0,$ $\xi \in \mathfrak{p},$
      \item $df(x) = 0.$
  \end{enumerate}
\end{corollary}

For the remaining part of this work, we fix $\beta = \mu_\mathfrak{p}(x).$ The negative gradient flow line of $f$ through $x\in X$ is the solution of the differential equation

\[ \left\{ \begin{array}{ll}
         \dot{x}(t) = -\beta_X(x(t)), \quad t\in \mathbb{R} \\
        x(0) = x.\end{array} \right. \]

The $G$-orbits are invariant under the gradient flow.
\begin{lemma}\label{Gradient}
Let $g: \mathbb{R} \rightarrow G$ be the unique solution of the differential equation
\[ \left\{ \begin{array}{ll}
         g^{-1}\dot{g}(t) = \beta_X(x(t)) \\
        g(0) = e,\quad \text{where $e$ is the identity of $G$}.\end{array} \right. \]

Then, $$x(t) = g^{-1}(t)x$$ for all $t\in \mathbb{R}.$
\end{lemma}

\begin{proof}
Define $y : \mathbb{R} \to X$ by $$y(t) = g^{-1}(t)x.$$ Since $\dot{g^{-1}} = -g^{-1}\dot{g}g^{-1}$ and $g^{-1}\dot{g} = \beta_X(x),$ it follows that
$$\dot{y} = -g^{-1}\dot{g}g^{-1}x =  -\beta_X(g^{-1}x) = -\beta_X(y(t))$$ and
$$y(0) = (g(0))^{-1}x = e^{-1}x = x.$$

Hence $x(t) = y(t) = g^{-1}(t)x$ for all $t\in \mathbb{R}.$

\end{proof}

The proof of the following Theorem is based on the Lojasiewicz gradient inequality, which holds in general for analytic gradient flows. A proof for the case of an action of a complex reductive group is given in \cite{Salamon}.
\begin{theorem}\label{teo}
Let $x_0\in X$ and $x: \mathbb{R} \rightarrow X$ be the negative gradient flow line of $f$ through $x_0$. There exist positive constants $\alpha,$ $C,$ $\psi,$ and $\frac{1}{2} < \gamma < 1$ such that

$$x_\infty := \lim_{t \rightarrow \infty} x(t) $$
exists. Moreover, there exist a constant $T > 0$ such that for any $t > T,$
\begin{align*}
    d(x(t), x_\infty) & \leq \int_t^\infty |\dot{x}(s)|ds\\
    &\leq \frac{\alpha}{1 - \gamma}(f(x(t)) - f(x_\infty))^{1-\gamma}\\
    & \leq \frac{C}{(t - T)^\psi}.
\end{align*}
\end{theorem}
\begin{proof}
Let $X = Z.$ Using the Marle-Guiliemin-Sternberg local normal form, the moment map is locally real analytic.
Since  $\mu_\mathfrak{p} = \pi_\mathfrak{p}\circ i\mu,$ where $\pi_\mathfrak{p} :i\mathfrak{u} \rightarrow \mathfrak{p}$ is the orthogonal projection, it follows that $\mu_\mathfrak{p}$ is locally real analytic. This implies that $f = \frac{1}{2}\parallel \mu_\mathfrak{p} \parallel^2 :Z \rightarrow \mathbb{R}$ satisfies the Lojasiewicz gradient inequality. By Lemma \ref{nmg}, the gradient of $f:Z \rightarrow \mathbb{R}$ coincide with the gradient of $f: X \rightarrow \mathbb{R}.$ 
Hence $f|_X$ also satisfies Lojasiewicz gradient inequality: there exists constants $\delta > 0,$ $\alpha > 0,$ and $\frac{1}{2} < \gamma < 1$ such that, for every critical value $a$ of $f$ and every $x\in X,$
\begin{equation}\label{loj}
    |f(x)-a| < \delta \qquad \implies |f(x)-a|^\gamma \leq \alpha|\triangledown f(x)|.
\end{equation}

Let $x : \mathbb{R} \rightarrow X$ be a nonconstant negative gradient flow line of $f.$
$$
a = \lim_{t \to \infty}f(x(t))
$$ is a critical value of $f.$ Choose a constant $T > 0$ such that $a < f(x(t)) < a + \delta$ for $t \geq T.$ Then, for $t\geq T,$
$$
\frac{d}{dt}(f(x)-a)^{1-\gamma} = (1-\gamma)(f(x)-a)^{-\gamma}|\triangledown f(x)| \geq \frac{1-\gamma}{\alpha}|\dot{x}|.
$$

Integrating the inequality over the interval $[t, \infty)$ gives
\begin{equation}\label{equ}
\int_t^\infty|\dot{x}(s)|ds \leq \frac{\alpha}{1-\gamma}(f(x(t))-a)^{1-\gamma} \qquad \text{for} \; \; t \geq T.
\end{equation}

This shows that $$x_\infty := \lim_{t\to \infty} x(t)$$ exists and it is a critical point of $f$ and hence satisfies $\mu_\mathfrak{p}(x_\infty)_X = 0.$

Set $\xi(t) = (f(x(t))-a)^{1-2\gamma}.$
$$\dot{\xi}(t) = (2\gamma - 1)(f(x(t))-a)^{-2\gamma}|\triangledown f(x(t))|^2\geq \frac{2\gamma-1}{\alpha^2} \qquad \text{for} \; \; t\geq T.
$$ Which implies that
$$
\xi(t) \geq \frac{2\gamma-1}{\alpha^2}(t-T) \qquad \text{for} \; \; t\geq T.
$$
Hence
$$(f(x(t))-a)^{1-\gamma} = \xi(t)^{-\frac{1-\gamma}{2\gamma-1}}\leq \left(\frac{2\gamma-1}{\alpha^2}(t-T)\right)^{-\frac{1-\gamma}{2\gamma-1}} \qquad \text{for} \; \; t\geq T.
$$
Thus
$$
\frac{\alpha}{1-\gamma}(f(x(t))-a)^{1-\gamma}\leq \frac{c}{(t-T)^\psi}, \quad \psi := \frac{1-\gamma}{2\gamma-1}, \quad c := \frac{\alpha}{1-\gamma}\left(\frac{\alpha^2}{2\gamma-1}\right)^\psi
$$ and by (\ref{equ}) the result follows.

\end{proof}

\subsection{Stratifications of the Norm Square of the Gradient map.}\label{ssss}
We recall the stratification theorem for actions of reductive group. First, we define a stratification of $X.$ For details see \cite{heinzner-schwarz-stoetzel}.

Given a maximal subalgebra $\mathfrak{a}\subset \mathfrak{p},$ we pick $\mathfrak{a}_+ \subset \mathfrak{a}$ a positive Weyl-chamber. Let $f: X \to \mathbb{R}$ be the norm square of the gradient map $\mu_\mathfrak{p}.$ i.e.,
$$
f(x) := \frac{1}{2}\parallel \mu_\mathfrak{p}(x) \parallel^2,
$$ where $\parallel \cdot \parallel$ denotes the norm functions associated to an $\mathrm{Ad}(K)$-invariant scalar product $\langle\cdot , \cdot\rangle.$ Let
$C$ denote the critical set of $f$, $\mathfrak{B} := \mu_\mathfrak{p}(C)$ and $\mathfrak{B}_+ := \mathfrak{B}\cap \mathfrak{a}_+.$

Let $X^{ss} := \{x\in X : \overline{G\cdot x} \cap \mu_\mathfrak{p}^{-1}(0) \neq \emptyset \}.$ For $\beta \in \mathfrak{B}_+$, following the notation introduced in \cite{heinzner-schwarz-stoetzel},  set

\begin{align*}
    & X|_{\parallel \beta\parallel^2} := \{x\in X : \overline{\text{exp}(\mathbb{R}\beta)\cdot x} \cap (\mu_\mathfrak{p}^\beta)^{-1}(\parallel \beta \parallel^2) \neq \emptyset \}\\
    & X^\beta := \{x\in X : \beta_X(x) = 0 \}\\
   & X^\beta|_{\parallel \beta \parallel^2} := X^\beta \cap X|_{\parallel \beta\parallel^2}\\
   & X^{\beta +}|_{\parallel \beta\parallel^2} := \{x\in X|_{\parallel \beta\parallel^2}: \lim_{t\to -\infty}\text{exp}(t\beta)\cdot x \; \text{exists and it lies in} \; X^\beta|_{\parallel \beta\parallel^2}\}
\end{align*}

The set $X^{\beta +}|_{\parallel \beta\parallel^2}$ is $G^{\beta +}$-invariant. $\mu_{\mathfrak{p}^\beta}$ is a gradient map of the $G^\beta$-action on $X^{\beta +}|_{\parallel \beta\parallel^2}.$ Set $$\widehat{\mu_{\mathfrak{p}^\beta}} := \mu_{\mathfrak{p}^\beta} - \beta.$$ Since $\beta$ is in the center of $\mathfrak{g}^\beta$ and $G^\beta$ is a compatible subgroup of $(U^\beta)^\C = (U^\C)^\beta, $ it is a gradient map too. Let
$$
S^{\beta +} := \{x \in X^{\beta +}|_{\parallel \beta\parallel^2}: \overline{G^\beta \cdot x} \cap \mu_{\mathfrak{p}^\beta}^{-1}(\beta) \neq \emptyset \}.
$$ The set $S^{\beta +}$ coincides with the set of semistable points of the group $G^\beta$ in $X^{\beta +}|_{\parallel \beta\parallel^2}$ after shifting.

\begin{definition}
The $\beta$-stratum of $X$ is given by $S_\beta := G\cdot S^{\beta +}.$
\end{definition}

\begin{theorem}\label{Stratification}
(Stratification Theorem)\cite[7.3]{heinzner-schwarz-stoetzel}. Suppose $X$ is a compact $G$-invariant submanifold of $Z.$ Then $\mathfrak{B}_+$ is finite and
$$
X = \bigsqcup_{\beta \in \mathfrak{B}_+}S_\beta.
$$
Moreover
$$
\overline{S_\beta}\subset S_\beta \cup \bigcup_{|\gamma| > |\beta|}S_\gamma.
$$
\end{theorem}

\begin{proposition}\cite[6.12]{heinzner-schwarz-stoetzel} \label{lemm}
If $z\in X$ satisfies $$f(z) = max_{x\in X}f(x).$$ Then $G\cdot z = K\cdot z$ and so it is closed orbit.
\end{proposition}
If $v\in T_p X$, then $|v|=\sqrt{(v,v)}$, where $(\cdot,\cdot)$ is the scalar product induced by the K\"ahler form $\omega$.
The following proposition give the Hessian of $f.$
\begin{proposition}\label{Hessian}
\cite[Prop. 2.5, 2]{heinzner-schwarz-stoetzel}. Let $v\in T_xX$ be an eigenvector of $\beta \in \mathfrak{p}_x$ with eigenvalue $\lambda(\beta).$ Let $\gamma(t)$ be a smooth curve in $X$ with $\gamma(0) = x$ and $\dot{\gamma}(0) = v.$ Then if $x$ is a critical point of $f,$
\begin{equation}\label{Hess}
\frac{d^2}{dt^2}(f\circ\gamma)(0) = \lambda(\beta)|v|^2 + \parallel d\mu_\mathfrak{p}(x)v \parallel^2
\end{equation}
\end{proposition}

The Hessian of $f$ at critical points satisfies the following:

\begin{proposition}\label{Hessian comp}
\cite[Prop. 6.6]{heinzner-schwarz-stoetzel}
Let $x\in C$ be a critical point of $f : X \to \mathbb{R}$ and let $S_\beta$ be the associated stratum. Let $H_x(f)$ denote the Hessian of $f$ at $x.$ Then
\begin{enumerate}
    \item $H_x(f) = 0$ on $T_x(K\cdot x),$
    \item $H_x(f) > 0$ on $\mathfrak{p}^\beta\cdot x + \lier^{\beta+}\cdot x,$ where $\lier^{\beta+}$ is the Lie algebra of $R^{\beta+},$
    \item $H_x(f)\geq 0$ on $T_x(S_\beta) = \lieg\cdot x + T_x(S^{\beta+}) = \liek\cdot x + T_x(S^{\beta+})$ and
    \item $H_x(f) < 0$ on $T_x(S_\beta)^\bot = (\lieg\cdot x)^\bot \cap (T_x(S^{\beta+}))^\bot = (\liek\cdot x)^\bot \cap (T_x(S^{\beta+}))^\bot.$
\end{enumerate}
\end{proposition}

\begin{remark}\label{Hessian comp1}
\cite[6.7]{heinzner-schwarz-stoetzel} The tangent space $T_x(G\cdot x)$ decompose to $T_x(G\cdot x) = T_x(K\cdot x)\oplus \liep^\beta\cdot x \oplus \lier^{\beta+}\cdot x.$ This follows from the decomposition $G = K G^{\beta+},$ $G^{\beta+} = G^{\beta}R^{\beta+},$ the identity $K\cap G^{\beta+} = K^\beta$ and the fact that $G^\beta$ acts on $X^\beta$ whereas $R^{\beta+}$ acts on the fibers of $\liep^{\beta+}.$ Thus the behaviour of $H_x(f)$ on $T_x(G\cdot x)$ is precisely described by Proposition \ref{Hessian comp}.
\end{remark}

From the proof of Theorem \ref{teo}, it was observed that $f$ satisfies the Lojasiewicz’s gradient inequality. As an application, we prove a well-known result that the stratum corresponding to the minimum of $f$ is open.
\begin{theorem}
If $\beta \in \mathfrak{B}$ is such that $\frac{1}{2}\parallel \beta\parallel^2$ is a minimum value of $f.$ Then the corresponding stratum $S_\beta$ is open in $X.$
\end{theorem}

\begin{proof} From the proof of Theorem \ref{teo}, there exists constants $\delta > 0,$ $\alpha > 0,$ and $\frac{1}{2} < \gamma < 1$ such that, for every critical value $a$ of $f$ and every $x\in X,$
\begin{equation}
    |f(x)-a| < \delta \qquad \implies |f(x)-a|^\gamma \leq \alpha|\triangledown f(x)|.
\end{equation} In particular, if $x$ is a critical point, then $f(x) = a.$ Since $X$ is compact, by  Theorem \ref{Stratification}, $$X = \bigsqcup_{i = 1}^kS_{\beta_i},$$ where $\beta_i \in \mathfrak{B}_+.$ We may assume that $\beta = \beta_1$ and so, $\parallel \beta_j \parallel > \parallel \beta \parallel$ for any $j = 2, \cdots, k.$
Let $$0< \delta' < \text{min} \left(\delta, \frac{\parallel \beta_2 \parallel^2 - \parallel \beta\parallel^2}{2}, \cdots, \frac{\parallel \beta_k\parallel^2 - \parallel \beta\parallel^2}{2}\right).$$ Let
$$
U = \{x\in X : |f(x) - \parallel \beta\parallel^2|< \delta'\}.
$$ Let $x_0\in U$ and let $x(t)$ the gradient flow of $-\nabla f$ through $x_0.$ Therefore,
\begin{align*}
    f(x(t)) \leq f(x_0) \leq |f(x_0) - \parallel \beta\parallel^2| + \parallel \beta\parallel^2\\
    & < \delta' + \parallel \beta\parallel^2\\
    & < \frac{\parallel \beta_j \parallel^2 - \parallel\beta\parallel^2}{2} + \parallel \beta\parallel^2 \quad j = 2, \cdots, k\\
    & < \frac{\parallel \beta_j \parallel^2 + \parallel \beta\parallel^2}{2}\leq \parallel \beta_j \parallel^2, \quad j = 2, \cdots, k.
\end{align*}
Therefore, $f(x_\infty) = \parallel \beta\parallel^2$. This implies that $U \subset S_\beta$ and so, using standard arguments of the gradient flow,  $S_\beta$ is open.
\end{proof}

As in \cite{E. Lerman}, we have the following deformation retraction.
\begin{theorem}[Retraction Theorem] \label{retraction-theorem}
Let $\beta \in \mathfrak{a}_+$ be a critical value of $f.$ Let $S_\beta$ be the stratum associated to $\beta.$ Let $\phi_t(x)$ denote the gradient flow of $-\nabla f.$ Then there exist a $K$-equivariant strong deformation retraction of $S_\beta$ onto $S_\beta \cap \mu_\mathfrak{p}^{-1}(K\cdot \beta)$ given by
$$[0, \infty] \times S_\beta \to S_\beta \cap \mu_\mathfrak{p}^{-1}(K\cdot \beta), \quad (t, p) \mapsto \phi_t(p),$$ and $$\phi_\infty(p) = \lim_{t\to +\infty}\phi_t(p).$$
\end{theorem}

\section{Kempf-Ness Function}\label{Kempf-Ness Function}

Given $G$ a real reductive group which acts smoothly on $Z;$ $G = K\text{exp}(\mathfrak{p}),$ where $K$ is a maximal compact subgroup of $G.$ Let $X$ be a $G$-invariant locally closed submanifold of $Z.$ As Mundet pointed out in \cite{MUNDET}, there exists a function $\Phi : X \times G \rightarrow \mathbb{R},$ such that
$$
\langle \mu_\mathfrak{p}(x), \xi\rangle = \frac{d}{dt}\bigg \vert_{t=0}\Phi (x, \text{exp}(t\xi)), \qquad \xi \in \mathfrak{p},
$$ and satisfying the following conditions:

\begin{enumerate}
    \item For any $x\in X,$ the function $\Phi (x, .)$ is smooth on $G.$
    \item The function $\Phi(x, .)$ is left-invariant with respect to $K,$ i.e., $\Phi(x, kg) = \Phi(x, g).$
    \item For any $x\in X,$ $v\in \mathfrak{p}$ and $t\in \mathbb{R};$

    $$\frac{d^2}{dt^2}\Phi (x, \text{exp}(tv)) \geq 0.$$
    Moreover: $$\frac{d^2}{dt^2}\Phi (x, \text{exp}(tv)) = 0$$ if and only if exp$(\mathbb{R}v)\subset G_x.$

    \item For any $x\in X,$ and any $g, h \in G;$
    $$\Phi(x, hg) = \Phi(x, g) + \Phi(gx, h).$$ This equation is called the cocycle condition. The proof is given in \cite{LM}, see also \cite{bgs}.
\end{enumerate}

The function $\Phi : X \times G \rightarrow \mathbb{R}$ is called the Kempf-Ness function for $(X, G, K).$

Let $M = G/K$ and $\pi : G \rightarrow M.$ $M$ is a symmetric space of non-compact type \cite{borel-ji-libro}. By $(b),$ $\Phi(x, kg) = \Phi(x, g),$ the function $\Phi$ descend to $M$ as $$\Phi : X \times M \rightarrow \mathbb{R};$$
\begin{equation}
\Phi(x,\pi(g)) = \Phi(x,g^{-1}).
\end{equation}
In this paper we have fixed an $\mathrm{Ad}(U)$-invariant scalar product $\langle \cdot,\cdot \rangle$ on $\liu^\C$. We recall that $\langle \liu, i\liu \rangle =0$ and the multiplication by $i$ defines an isometry from $\liu$ onto $i\liu$. Hence $\lieg =\liek \oplus \liep$ is an orthogonal splitting with respect to $\langle \cdot, \cdot \rangle $. Equip $G$ with the unique left invariant Riemannian metric which agree with the scalar product  $\langle \cdot,\cdot \rangle$ on the tangent space $\lieg=\mathfrak k \oplus \liep$ of $G$ at $e$. This metric is $\mathrm{Ad}(K)$ invariant and so it induces a $G$ invariant Riemannian metric of nonpositive curvature on $M$ \cite{helgason}.
\begin{lemma}\label{differential}
For $x\in X$, let $\Phi_x(g  K) = \Phi(x,g^{-1}).$ The differential of $\Phi_x$ is given as
$$
d(\Phi_x)_{\pi(g)}(v_x) = -\langle \mu_\mathfrak{p}(g^{-1}x), \xi\rangle
$$ where, $v_x(g) = \mathrm d (\pi \circ L_g )_e (\xi)$ and $\xi \in \mathfrak{p}.$ Therefore, $\nabla \Phi_x(\pi(g)) =  -
\mathrm d ( \pi \circ L_g )_e (\mu_\mathfrak{p}(g^{-1}x))$
\end{lemma}

\begin{proof}
Let $\pi(g)\in M,$ $\xi \in \mathfrak{p}$ and $v_x\in T_{\pi(g)}G/K.$ There exist $\xi \in \mathfrak{p}$ such that

$$v_x = \frac{d}{dt}\bigg \vert_{t=0} g\text{exp}(t\xi)K.$$

Take
$$ \gamma(t) = \pi(g\text{exp}(t\xi)), \qquad t\in [a,b], \; \xi \in \mathfrak{p}.$$
Then 
$v_x = (d\pi)_g((dL_g)(\xi)),$

\begin{align*}
    d(\Phi_x)_{\pi(g)}(v_x) &= \frac{d}{dt}\bigg \vert_{t = 0}\Phi(x, \gamma(t))\\
    & = \frac{d}{dt}\bigg \vert_{t = 0}\Phi(x, \pi(g\text{exp}(t\xi))\\
    & = \frac{d}{dt}\bigg \vert_{t = 0}\Phi(x, \text{exp}(-t\xi)g^{-1}) \qquad (\mbox{by the deifintion of} \,\, \Phi)\\
    & = \frac{d}{dt}\bigg \vert_{t = 0} [\Phi(x, g^{-1}) + \Phi(g^{-1}x, \text{exp}(-t\xi)) ] \qquad (\mbox{by condition (d)})\\
    & = \frac{d}{dt}\bigg \vert_{t = 0} \Phi(g^{-1}x, \text{exp}(-t\xi))\\
    & = -\langle \mu_\mathfrak{p}(g^{-1}x), \xi\rangle.
 \end{align*}
This implies that
$$ d(\Phi_x)_{\pi(g)}(v_x) = -\langle \mu_\mathfrak{p}(g^{-1}x), \xi\rangle. 
$$
\end{proof}
We denote by the same symbol  $\Phi_x : G \rightarrow \mathbb{R}$, the function $\Phi_x (g)=\Phi(x,g^{-1})$. It is called the Kempf-Ness function at $x$.
Define $\phi_x: G\to G\cdot x$ as follows $$\phi_x(g) = g^{-1}x.$$
\begin{lemma}\label{differential2}
The map $\phi_x$ intertwines the gradient of $\Phi_x :G \lra \R$ and the gradient of $f.$ i.e., $\forall g\in G$
$$
d(\phi_x)_g\nabla\Phi_x = \nabla f((\phi_x(g)).
$$
\end{lemma}

\begin{proof}
Let $\beta\in\mathfrak{p}.$ Since
$$\phi_x(g\text{exp}(t\beta)) = \text{exp}(-t\beta)g^{-1}x,$$
we have
$$(d\Phi_x)_g(dL_g(\beta)) = -\beta_X(g^{-1}x).$$ The result follow from taking $\beta = \mu_\mathfrak{p}(x).$
\end{proof}
\begin{remark}
A smooth curve $\gamma(t):= \pi \circ g : \mathbb{R} \to M$ is a negative gradient flow line of $\Phi_x$ if and only if the smooth curve $g : \mathbb{R} \to M$ satisfies $g^{-1}\dot{g} = \mu_\mathfrak{p}(g^{-1}x).$ Indeed, from Lemma \ref{differential}, 
$$\nabla \Phi_x(\pi(g)) = - \mathrm d (\pi \circ dL_g)_e (\mu_\mathfrak{p}(g^{-1}x))$$ and by Lemma \ref{differential2}, $\phi_x$ interwines the gradient of $\Phi_x$ with $\nabla f.$ Therefore, the gradient flow of $\Phi_x$ is such that $g^{-1}\dot{g} = \mu_\mathfrak{p}(g^{-1}x)$. On the other hand if $g: \mathbb{R}\to G$ satisfies $g^{-1}\dot{g} = \mu_\mathfrak{p}(g^{-1}x)$, the curve $\gamma = \pi \circ g(t)$ is a geodesic and
$$
\frac{d}{dt}(\Phi_x\circ \gamma) = -\langle \mu_\mathfrak{p} (g^{-1}x), g^{-1}\dot{g}\rangle.
$$
\end{remark}

We recall some result from Riemannian geometry. We refer the reader to Apendix A in \cite{Salamon} for further details. Suppose $M$ is an Hadamard manifold, i.e., connected, complete, simply-connected with non-positive curvature. Let $\gamma_1, \gamma_2 : [a, b] \to M$ be smooth curves and for each $t\in [a, b],$ $\gamma(s, t): [0, 1] \to M$ be the unique geodesic such that $$
\gamma(0, t) = \gamma_1(t), \qquad \gamma(1,t) = \gamma_2(t).
$$
Define the function $\rho: [a, b] \to \mathbb{R}$ by
$$\rho (t) := d(\gamma_1(t), \gamma_2(t)).$$ The following Lemma is proved in \cite[Lemma A.2 ]{Salamon}
\begin{lemma}\label{rmm} Suppose $f: M\to \mathbb{R}$ is a smooth function that is convex along geodesics. Let $\gamma_1, \gamma_2 : \mathbb{R} \to M$ be the negative gradient flow lines of $f,$ and let $\gamma$ and $\rho$ be as defined above. Then $\rho$ is nonincreasing and, if $\rho(t) \neq 0,$ then
\begin{equation}\label{rho}
    \dot{\rho}(t) = -\frac{1}{\rho(t)}\int_0^1\frac{\partial^2}{\partial s^2}(f\circ \gamma)(s,t)ds.
\end{equation}
\end{lemma}

\begin{theorem}\label{ConvexKempf}
Let $\Phi_x : M \to \mathbb{R}.$ Then
\begin{enumerate}
    \item $\Phi_x$ is a Morse-Bott function and it is convex along geodesics.
    \item If $\gamma : \mathbb{R}\to M$ is a negative gradient flow of $\Phi_x,$ then, $$\lim_{t\to \infty}\Phi_x(\gamma(t)) = \text{inf}_{x\in M}\Phi_x.$$
\end{enumerate}

\end{theorem}

\begin{proof}
(a): By lemma \ref{differential}, $g\in G$ is a critical of $\Phi_x$ if and only if $$\mu_\mathfrak{p}(g^{-1}x) = 0.$$
\begin{equation}
\text{Crit}(\Phi_x) = \{\pi(g)\in M : \mu_\mathfrak{p}(g^{-1}x) = 0\}.
\end{equation}
The next is to show that the $\text{Crit}(\Phi_x)$ is a submanifold of $M.$ To do this, the Hessian of the function is computed along geodesics. The geodesic on $M$ passing through $\pi(g)$ in the direction $v = d\pi_gg\xi$ has the form $\pi(g\text{exp}(t\xi))$ \cite{LM}. Hence, $M$ is complete and by the Hadamard theorem, $$\mathfrak{p} \rightarrow M, \qquad \xi \mapsto \pi(g\text{exp}(\xi))$$ is a diffeomorphism. This implies that $\pi(g\text{exp}(\xi))\in \text{Crit}(\Phi_x)$ if and only if $\mu_\mathfrak{p}(g^{-1}\text{exp}(-\xi)x) = 0.$
$$
\text{Hess}(\Phi_x) = d^2(\Phi_x)_{\pi(g)}(v), \qquad \pi(g)\in \text{Crit}(\Phi_x)
$$
\begin{align*}
d^2(\Phi_x)_{\pi(g)}(v) & = \frac{d^2}{dt^2}\bigg \vert_{t = 0}\Phi(x, \gamma(t))\\
    & = \frac{d^2}{dt^2}\bigg \vert_{t = 0}\Phi(x, \pi(g\text{exp}(-t\xi))\\
    & = \frac{d^2}{dt^2}\bigg \vert_{t = 0}\Phi(x, g\text{exp}(-t\xi)K)\\
    & = \frac{d^2}{dt^2}\bigg \vert_{t = 0}\Phi(x, \text{exp}(t\xi)g^{-1})\\
    & = \frac{d^2}{dt^2}\bigg \vert_{t = 0} [\Phi(x, g^{-1}) + \Phi(g^{-1}x, \text{exp}(t\xi)) ]\\
    & = \frac{d^2}{dt^2}\bigg \vert_{t = 0}\Phi(g^{-1}x, \text{exp}(t\xi) \geq 0 \qquad (\mbox{by condition (c)})
 \end{align*}
Moreover, $$ \frac{d^2}{dt^2}\bigg \vert_{t = 0}\Phi(g^{-1}x, \text{exp}(t\xi) = 0$$ if and only if exp$(t\xi) \subset G_{g^{-1}x}.$ Hence,
$$
\text{Crit}(\Phi_x) = \{\pi(g\text{exp}t\xi)\in M : \text{exp}t\xi \subset G_{g^{-1}x}\},
$$
which is a submanifold and the kernel of the Hessian. Therefore, $\Phi$ is a Morse-Bott function and since the Hessian is non-negative along geodesics, it is convex along geodesics.

(b). Let $\gamma_1, \gamma_2$ be negative gradient flow of $\Phi_x.$ There exists $g_1, g_2: \mathbb{R}\to G$ such that $\gamma_1 = \pi(g_1(t))$ and $\gamma_2 = \pi(g_2(t)).$ Let $\beta: \mathbb{R}\to \mathfrak{p}$ and $k: \mathbb{R}\to K$ be such that $g_2(t) = g_1(t)\text{exp}(\beta(t))k(t).$ Define $H: \mathbb{R}\times \mathbb{R}\to M$ by
$$
H(t,s) =  \pi(g_1(t)\text{exp}(s\beta(t))).
$$
The curve $s\mapsto H(t, s)$ is the unique geodesic joining $\gamma_1$ and $\gamma_2.$ By Lemma \ref{rmm} the function $\rho(t) = d_M(\gamma_1(t), \gamma_2(t)) = \parallel\beta(t)\parallel$ is nonincreasing.

Assume that Crit$(\Phi_x)$ is not empty. Hence we may suppose that $\mu_\mathfrak{p}(g_1(0)^{-1}x) = 0$. This implies that the curve $\gamma_1$ is constant. Since $\rho$ is nonincreasing, the image of  $\gamma_2$ is contained in a compact subset of $M.$ Since $\Phi_x$ is Morse-Bott, then $\gamma_2$ converges to a critical point of $\Phi_x.$ This implies that if the critical manifold of $\Phi_x$ is nonempty, then $\Phi_x$ has a global minimal and every negative flow line of $\Phi_x$ converges to a critical point. Now suppose Crit$(\Phi_x)$ is empty. Assume by contradiction that $$a := \lim_{t\to\infty}\Phi_x(\gamma_1(t)) \geq \text{inf}_M\Phi_x.$$ Then, $\Phi_x(\gamma_1(t))$ is bounded from below. We can choose $\gamma_2$ such that $\Phi(\gamma_2(0))< a.$ Since the function $\rho = \parallel \beta(t)\parallel$ is nonincreasing, there exists a constant $C > 0$ such that $\rho(t) = \parallel \beta(t)\parallel \leq C.$ Hence,
\begin{align*}
\frac{d}{ds}\bigg \vert_{s=0}\Phi(H(t,s)) &= (d\Phi_x)_{\gamma_1(t)} (\dot{H}(t,0))\\
& = -\langle \mu_\mathfrak{p}(g_1(t)^{-1}x), \beta(t)\rangle\\
& \geq - \parallel \mu_\mathfrak{p}(g_1(t)^{-1}x)\parallel \parallel \beta(t)\parallel\\
&\geq -C\parallel\mu_\mathfrak{p}(g_1(t)^{-1}x)\parallel.
\end{align*}
Since for a fixed $t,$ the function $s\to \Phi_x(H(t,s))$ is convex, this implies that the derivative $\frac{d}{ds}\Phi(H(t,s))$ increases. It follows that
\begin{align*}
    \Phi_x(\gamma_2(t)) &= \Phi_x(H(t,1))\\
    & = \Phi_x(H(t,0)) + \int_0^1\frac{d}{ds}\Phi(H(t,s)) \mathrm d s\\
    & \geq \Phi_x(\gamma_1(t)) - C\parallel \mu_\mathfrak{p}(g_1(t)^{-1}x)\parallel.
\end{align*} Since the function $\Phi_x(\gamma_1(t))$ is bounded below and $\frac{d}{dt}\Phi_x(\gamma_1(t)) = -\parallel \mu_\mathfrak{p}(g_1(t)^{-1}x)\parallel^2,$ there exists a sequence $t_i\to \infty$ such that $\lim_{i\to \infty}\parallel \mu_\mathfrak{p}(g_1(t_i)^{-1}x)\parallel^2 = 0.$ This implies that
$$
\lim_{i\to\infty}\Phi_x(\gamma_2(t_i))\geq \lim_{i\to\infty}\Phi_x(\gamma_1(t_i)) = a.
$$ This is a contradiction since by assumption $\Phi_x(\gamma_2(t)) < a$ and so $\lim_{t\to\infty}\Phi_x(\gamma_1(t)) < a.$

\end{proof}
The following result asserts that any critical points of $f$ in the same $G$-orbit in fact belong to
the same $K$-orbit. We use original ideas from \cite{Salamon} in a different context.
\begin{theorem}
\label{critt}
Let $x_0, x_1 \in X$ be critical points of the norm square $f.$ Then $$x_1 \in G\cdot x_0 \implies x_1 \in K\cdot x_0.$$
\end{theorem}

\begin{proof}
Since $x_0, x_1 \in X$ are critical points of $f,$ then by Lemma \ref{nmg}
\begin{equation}
    \mu_\mathfrak{p}(x_0)_X = 0, \qquad \mu_\mathfrak{p}(x_1)_X = 0
\end{equation}
Suppose $x_1 \in G\cdot x_0.$ Let $g_0\in G$ such that
$$x_1 = g_0^{-1}x_0$$

and $g,h : \mathbb{R} \rightarrow G$ be defined by
$$
g(t) := \text{exp}(t\mu_\mathfrak{p}(x_0)), \qquad \qquad h(t) := g_0\text{exp}(t\mu_\mathfrak{p}(x_1)).
$$
Since $\mu_\mathfrak{p}(x_0)_X = 0$, then $g(t)^{-1}x_0 = \text{exp}(-t\mu_\mathfrak{p}(x_0))x_0 = x_0$. Similarly $$h(t)^{-1}x_0 = \text{exp}(-t\mu_\mathfrak{p}(x_1))g_0^{-1}x_0 = g_0^{-1}x_0 = x_1$$
for all $t$. Thus $g(t)$ and $h(t)$ satisfy the differential equation $g^{-1}\dot{g} = \mu_\mathfrak{p}(g^{-1}x_0).$ These implies that the curves $\gamma_1 := \pi \circ g$ and $\gamma_2 := \pi \circ h$ are geodesics and are negative gradient flow lines of the Kempf-Ness function. Define $\xi(t)\in \mathfrak{p}$ and $k(t) \in K$ so that
$$
h(t) := g(t)\text{exp}(\xi(t))k(t).
$$

$$
x_1 = h(t)^{-1}g(t)x_0 = k(t)^{-1}\text{exp}(-\xi(t))x_0.
$$

Let
$$
\rho(t) := d_M(\gamma_1(t), \gamma_2(t)) = \parallel \xi(t)\parallel, \quad \text{for all t.}
$$

If $\rho \equiv 0,$ then $\xi(t) = 0$ for all $t$ and $$
x_1 = k(t)^{-1}x_0
$$ and this means that $x_1\in K \cdot x_0.$ Otherwise, for each $t$, let $\gamma(s,t): [0, 1] \to M$ defined by
$$
\gamma(s, t) := \pi(g(t)\text{exp}(s\xi(t)))
$$ be the unique geodesic. Note that $\gamma(0,t) = \gamma_1(t)$ and $\gamma(1,t) = \gamma_1(t).$

By equation \ref{rho}, we have

\begin{align}\label{eqq}
    \dot{\rho}(t) &= -\frac{1}{\rho(t)}\int_0^1\frac{\partial^2}{\partial s^2}\Phi_{x_0}(g(t)\text{exp}(s\xi(t)) \mathrm d s\\
    &= -\frac{1}{\rho(t)}\int_0^1 ( \xi(t)_X(\text{exp}(s\xi(t))x_0), \xi(t)_X(\text{exp}(s\xi(t))x_0) \mathrm d s.\label{eqqq}
\end{align}
Choose a sequence $t_n\to \infty$ such that the limits

$$
\lim_{n\to \infty} \dot{\rho}(t_n) = 0, \quad \xi_\infty :=\lim_{n\to \infty}\xi(t_n), \quad k_\infty := \lim_{n\to \infty}k(t_n)
$$ exist. By (\ref{eqqq}) and since $\lim_{n\to \infty} \dot{\rho}(t_n) = 0,$ $\xi_{\infty X}(x_0) = 0.$ Then,

$$
x_1 =  \lim_{n\to \infty}k(t_n)^{-1}\text{exp}(-\xi(t_n))x_0 = k_\infty^{-1}\text{exp}(-\xi_\infty)x_0 = k_\infty^{-1}x_0.
$$ Hence $x_1\in K\cdot x_0$
\end{proof}

The following theorems also hold in the real case.
\begin{theorem}\label{MLT} Let $x_0\in X$ and $x: \mathbb{R} \rightarrow X$ the negative gradient flow line of $f$ through $x_0.$ Define $x_\infty := \lim_{t\to\infty}x(t).$ Then
$$
\parallel \mu_\mathfrak{p}(x_\infty)\parallel = \text{inf}_{g\in G}\parallel \mu_\mathfrak{p}(gx_0)\parallel.
$$
Moreover, the $K$-orbit of $x_\infty$ depends only on the $G$-orbit of $x_0.$
\end{theorem}
\begin{proof}
The limit $x_\infty$ exists by Theorem \ref{teo}. The solution of the negative gradient flow line of $f$ through $x_0$ by Lemma \ref{Gradient} is given by
$$
x(t) = g(t)^{-1}x_0,
$$
where $g : \mathbb{R} \to G$ is the solution of

\[ \left\{ \begin{array}{ll}
         g^{-1}\dot{g}(t) = \beta_X(x(t)) \\
        g(0) = e,\quad \text{where $e$ is the identity of $G$}.\end{array} \right. \]

Fix an element $g_0\in G$ and let $y : \mathbb{R}\to X$ and $h : \mathbb{R}\to G$ be the solutions of the differential equations
$$
\dot{y} = -\beta_X(y(t)), \qquad y(0) = g_0^{-1}x_0,
$$ and
$$
h^{-1}\dot{h} = \beta_X(y(t)), \qquad h(0) = g_0.
$$
Define $\xi(t) \in \mathfrak{p}$ and $k(t)\in k$ by
$$
h(t) =: g(t)\text{exp}(\xi(t))k(t).
$$

By Lemma \ref{Gradient},
$$
y(t) = h(t)^{-1}x_0 = k(t)^{-1}\text{exp}(-\xi(t))x(t), \quad \forall t\in \mathbb{R}.
$$
Let $d_M : M \times M \to [0,\infty)$ be the distance function of the Riemannian metric on $M.$ $\gamma_1 := \pi \circ g$ and $\gamma_2 := \pi \circ h$ are geodesics and are negative gradient flow lines of the so called Kempf-Ness function. Since $M$ is simply connected with nonpositive sectional curvature. Then
$$
d_M(\gamma_1(t), \gamma_2(t)) = \parallel\xi(t)\parallel
$$ is nonincreasing. Hence there exist a sequence $t_n\to \infty$ such that the limits

$$
\xi_\infty :=\lim_{n\to \infty}\xi(t_n), \quad k_\infty := \lim_{n\to \infty}k(t_n)
$$ exist. Hence

$$
y_\infty = \lim_{t\to \infty}y(t) = \lim_{t\to \infty}k(t)^{-1}\text{exp}(-\xi(t))x(t) = k_\infty^{-1}\text{exp}(-\xi_\infty)x_\infty.
$$ Which implies that $y_\infty$ and $x_\infty$ are critical points of the normed sqaure of the gradient map belonging to the same $G$-orbit. Hence they belong to the same $K$-orbit by Theorem \ref{critt} and therefore,
$$
\parallel \mu_\mathfrak{p}(x_\infty)\parallel = \parallel \mu_\mathfrak{p}(y_\infty)\parallel\leq \parallel \mu_\mathfrak{p}(g_0^{-1}x_0) \parallel.
$$

\end{proof}
\begin{theorem}\label{SNUT}. Let $x_0\in X$ and
$$
m := \text{inf}_{g\in G}\parallel \mu_\mathfrak{p}(gx_0)\parallel.
$$Then
$$
x,y\in \overline{G\cdot x_0}, \quad \parallel \mu_\mathfrak{p}(x)\parallel = \parallel \mu_\mathfrak{p}(y)\parallel = m \quad \implies y\in K\cdot x.
$$
\end{theorem}
\begin{proof}
The solution of the negative gradient flow line of $f$ through $x_0$ is given by
$$
x(t) = g(t)^{-1}x_0,
$$
Fix $g_0\in G$ and we know that the limit $x_\infty$ of $x(t)$ exists. Then by Theorem \ref{MLT},
$$
x_\infty \in \overline{G\cdot x}, \quad \parallel \mu_\mathfrak{p}(x_\infty)\parallel = m.
$$
Let $x \in \overline{G\cdot x}$ such that $\parallel \mu_\mathfrak{p}(x)\parallel = m.$

Choose a sequence $g_n\in G$ such that
$$
x = \lim_{n\to \infty}g_n^{-1}x_0
$$ and define $y_n : \mathbb{R}\to X$ and $x_i\in X$ by
$$
\dot{y_n} = -\beta_X(y_n), \quad y_n(0) = g_n^{-1}x_0, \quad x_n := \lim_{t\to \infty}y_n(t).
$$
Then from the estimate of Theorem \ref{teo}, there exists a constant $c > 0$ such that, for $n$ sufficiently large,

$$
d(x_n, g_n^{-1}x_0) \leq \int_0^\infty |\dot{y_n}(t)|dt \leq c(\parallel \mu_\mathfrak{p}(g_n^{-1}x_0)\parallel^2 - m^2)^{1-\alpha}.
$$
Since
$$
m = \parallel \mu_\mathfrak{p}(x)\parallel = \lim_{n\to \infty}\parallel \mu_\mathfrak{p}(g_n^{-1}x_0)\parallel,
$$ which implies that $x = \lim_{n\to \infty}x_n$ and $x_n\in K\cdot x_\infty$ for all $n$ by Theorem \ref{MLT}. Therefore, $x\in K\cdot x_\infty$ because the group orbit $K\cdot x_\infty$ is compact.
\end{proof}

Let $x\in X$ be a critical point of $f$ and $y : \mathbb{R}\to M$ be the unique solution of the equation
$$
\dot{y} = -\beta_X(y(t)), \quad y(0) = y_0\in X; \quad \beta = \mu_\mathfrak{p}(y).
$$
We define the stable manifold of the critical set $K\cdot x$ by
\begin{equation}\label{Stable manifold}
S(K\cdot x) := \{y_0 \in X : \lim_{t\to \infty}y(t) = kx\quad \text{for some}\quad k\in K\}.
\end{equation}
By Theorem \ref{teo}, $X$ is the union of these stable manifolds and each stable manifold is a union of $G$-orbits by Theorems \ref{MLT} and \ref{SNUT}. The stable manifolds of the gradient flow  have a structure close to a stratification by stable manifolds corresponding to a Morse-Bott function.

By Theorems \ref{teo}, \ref{MLT} and \ref{SNUT}, we have the following result. This result generalises Theorem 5.2 proved by Jablonski for the G-gradient map of a projective representation \cite{Jab}.
\begin{theorem}\label{corr}
  Let $x\in X$ be a critical point of $f$ and $S(K\cdot x)\subset X$ be as defined above. The following holds:
 \begin{enumerate}
     \item $X = \bigcup_{x\in \text{Crit}(f)}S(K\cdot x).$
     \item Let $y_o\in X.$ Then $y_0\in S(K\cdot x)$ if and only if
     \begin{equation}\label{eequa}
         x\in \overline{G\cdot y_0}, \quad \parallel \mu_\mathfrak{p}(x)\parallel = \text{inf}_{g\in G}\parallel\mu_\mathfrak{p}(gy_0)\parallel
     \end{equation}
     \item $S(K\cdot x)$ is a union of $G$-orbits.
 \end{enumerate}
\end{theorem}

\begin{proof}
(a) follows by Theorem \ref{teo}.

To proof (b); let $y : \mathbb{R}\to M$ be the unique solution of the equation
$$
\dot{y} = -\beta_X(y(t)), \quad y(0) = y_0\in X; \quad \beta = \mu_\mathfrak{p}(y)
$$ and $y_\infty := \lim_{t\to \infty}y(t).$
Then, by Lemma \ref{Gradient} and Theorem \ref{MLT}, we have
\begin{equation}\label{eeqqa}
y_\infty \in \overline{G\cdot y_0}, \quad \parallel \mu_\mathfrak{p}(y_\infty)\parallel = \text{inf}_{g\in G}\parallel \mu_\mathfrak{p}(gy_0)\parallel.
\end{equation}
From (\ref{Stable manifold}), $y_0\in S(K\cdot x)$ if and only if $y_\infty \in K\cdot x.$ Thus $y_0\in S(K\cdot x)$ implies (\ref{eequa}). Conversely, if $y_0$ satisfies (\ref{eequa}), then it follows from (\ref{eeqqa}
) and Theorem \ref{SNUT} that $y_\infty \in K\cdot x$ and hence, $y_0\in S(K\cdot x).$

To proof (c); From (ii) and the uniqueness in Theorem \ref{MLT} that $S(K\cdot x)$ is a union of $G$-orbits.
\end{proof}

\section{Convexity Properties of Gradient map}
In this section, we prove convexity properties of the gradient map. 
\subsection{The Abelian Convexity Theorem}\label{Convexity Property of Gradient map}
Suppose $X$ is compact and connected. Let $\beta \in \mathfrak{p}$ and let $$Y = \{z\in X : \text{max}_{x\in X}\mu_\mathfrak{p}^\beta = \mu_\mathfrak{p}^\beta(z)\}.$$

By Corollary \ref{slice-cor-2}, $Y$ is a smooth, possibly disconnected, submanifold of $X$.
\begin{lemma}\label{lemmm}
$Y$ is $G^{\beta +}$ invariant.
\end{lemma}
\begin{proof}

Let $g\in G$ and let $\xi \in \liep$. It is easy to check that
\[
(\mathrm{d} g)_p (\xi_X)=(\mathrm{Ad}(g)(\xi) )_X (gp ),
\]
and so $G^\beta$ preserves $X^\beta$. We claim that $Y$ is $G^\beta$-stable. In fact, by Lemma \ref{lemcomp} it follows that $G^\beta = K^\beta\text{exp}(\mathfrak{p}^\beta)$. $Y$ is $K^\beta$ invariant by $K$-equinvariant property of the gradient map. For each $y\in Y,$ let $\xi\in \liep^\beta$ and let $\gamma(t)= \text{exp}(t\xi)y$. Since $\beta_X (\gamma(t))=0$ it follows that $\mup^\beta (\gamma(t) )$ is constant and so $\text{exp} (t\xi)\cdot y\in Y.$ Therefore $Y$ is $G^\beta$ invariant.  By Lemma \ref{parabolicc}, $G^{\beta+} = G^\beta R^{\beta+}$ where $R^{\beta+}$ is connected and then unipotent radical of $G^{\beta+}$. By Proposition \ref{tangent}, $\lier^{\beta+} \subset V_{+}$, where $V_+$ is the sum of the eigenspaces of the
Hessian of $\mupb$ corresponding to positive eigenvalues. Hence $\lier^{\beta+}\cdot z \subset \mathfrak g_z$ for any $z\in Y.$ This implies that $R^{\beta+}$ acts trivially on $Y$ and the result follows.
\end{proof}
\begin{proposition}\label{closed-orbit-parabolic}
$Y$ contains a compact orbit of $G^{\beta+}$ which coincides with a $K^\beta$ orbit.
\end{proposition}
\begin{proof}
By Lemma \ref{lemmm}, $(G^{\beta})^o$ preserves any connected component of $Y$ and the restriction of $\mup$ on any connected component defines a gradient map with respect to $(G^{\beta})^o$ \cite{heinzner-schwarz-stoetzel}. By Corollary 6.11 in \cite{heinzner-schwarz-stoetzel} pag. $21$, $(G^{\beta+})^o$ has closed orbit which coincides with a $(K^{\beta})^o$ orbit. Since $G^{\beta}$ has a finite number of connected components and any connected component of $G^\beta$ intersects $K^\beta$, it follows that $G^{\beta}$ has a closed orbit which coincides with a $K^{\beta}$ orbit. This is also a closed orbit of $G^{\beta+}$ since $R^{\beta+}$ acts freely on $Y$, concluding the proof.
\end{proof}

Let $\mathfrak{a} \subset \mathfrak{p}$ be an Abelian subalgebra of $\mathfrak{p}$ and let $\pi_\mathfrak{a} : \mathfrak{p} \rightarrow \mathfrak{a}$
be the orthogonal projection onto $\mathfrak{a}.$ It is well known that $\mu_\mathfrak{a} := \pi_\mathfrak{a} \circ \mu_\mathfrak{p}$ is the gradient map
associated to $A = \text{exp}(\mathfrak{a}).$ Let $P = \text{Conv}(\mu_\mathfrak{a}(X)).$
It is well-known, see for instance \cite[Prop. $3$]{heinzner-schuetzdeller} and \cite[Prop. $3.1$]{LAP},  that  $\mu(X^A)$ is finite and $P$ is
is the convex hull of $\mu_\mathfrak{a}(X^A)$, 
where $
X^A=\{p\in X:\, A\cdot p=p\}
$.

Suppose that the $G$ action on $X$ has a unique compact orbit, which is a $K$ orbit \cite{heinzner-schwarz-stoetzel} denoted by $\mathcal O$. Let $\mathfrak{a}'$ be a maximal Abelian subalgebra containing $\mathfrak{a}.$ Since $\mu_\mathfrak{a}(\mathcal O) = \pi_\mathfrak{a}(\mu_{\mathfrak{a}'}(\mathcal O)),$ By a Theorem of Kostant \cite{kostant-convexity} it follows that $\mu_\mathfrak{a}(\mathcal O)$ is a polytope.
\begin{theorem}\label{convexity property}
Suppose the $G$-action on $X$ has a unique closed orbit $\mathcal O.$ Then $\mu_\mathfrak{a}(X) = \mu_\mathfrak{a} (\OO)$ and so it is a convex polytope.
\end{theorem}
\begin{proof}
Let $\xi \in \lia$. Then
\[
\sigma =\{\alpha \in P:\, \mathrm{max}_{\gamma \in P} \langle \gamma, \xi \rangle = \langle \alpha, \xi \rangle \}.
\]
is a face of $P$. Since $P$ is a polytope, any face of $P$  is exposed \cite{schneider-convex-bodies}. We claim that $$\mu_\lia^{-1}(\sigma) = Y,\quad \text{where}\quad Y = \{z\in X : \text{max}_{x\in X}\mu_\mathfrak{p}^\xi = \mu_\mathfrak{p}^\xi(z)\}. $$ 
In fact, it is easy to see that $Y\subset \mu_\lia^{-1}(\sigma).$ Suppose $z\in \mu_\lia^{-1}(\sigma),$ then $\mu_\lia(z) \in \sigma$ and $\mathrm{max}_{\gamma \in P} \langle \gamma, \xi \rangle = \langle \mu_\lia(z), \xi \rangle = \mu_\mathfrak{p}^\xi(z).$ Hence, $z\in Y.$ By Lemma \ref{lemmm}, $Y$ is $G^{\xi+}$-invariant. By Proposition \ref{closed-orbit-parabolic}, let $z\in Y$ be such that $G^{\xi+} \cdot z$ is compact. But $G = KG^{\xi+},$ then $G\cdot z$ is compact. Since the $G$ action on $X$ has a unique compact orbit, $G\cdot z = \OO.$ Therefore, 
$$\text{max}_{x\in P}\langle x, \xi \rangle = \text{max}_{x\in\mu_\mathfrak{a}(\OO)}\langle x, \xi \rangle.$$ 
By Proposition \ref{convex-criterium}, $\mu_\mathfrak{a}(\OO) = P.$ Hence, $\mu_\mathfrak{a}(X) = \mu_\mathfrak{a} (\OO)$ is a polytope.
\end{proof}
\begin{remark}
In the above assumption, applying the main Theorem in \cite{LAP}, the convex hull of $\mup(X)$ coincides with the convex hull of $\mup(\OO)$. Hence the convex hull of $\mup(X)$ is the convex hull of a $K$-orbit in $\liep$ and so a polar orbitope \cite{LA}.
\end{remark}
\subsection{Abelian convexity fron Non-Abelian convexity}
The Non-Abelian convexity theorem implies the Abelian convexity theorem. This is the purpose of this section.

Let $\mathfrak{a}\subset \mathfrak{p}$ be a maximal Abelian, $\mathfrak{a}_+$ positive Weyl Chamber. If $\lambda \in \mathfrak{a}_+,$ we denote by
$$
\roots_\la = \text{Conv}\{w\lambda : w\in W\},
$$ where $W = \frac{N_k(\mathfrak{a})}{C_k(\mathfrak{a})}$ is the Weyl-Group. If $G=U^\C$, then the following result is proved in \cite{gs}. The authors applied Kirwan's Theorem \cite{kirwan-convexity} for the action $U\times \mathbb T$ on the cotangent bundle $T^* U$, where $\mathbb T$ is a maximal torus of $U$. Our proof uses a result of Gichev \cite{gichev-polar}.
\begin{theorem}
If $S\subset \mathfrak{a}_+$ is a convex subset, then,
\begin{equation}\label{Convex}
S^\# = \bigcup\{\roots_\la : \la \in S\}
\end{equation} is convex subset of $\mathfrak{a}.$
\end{theorem}
\begin{proof}
Let
$$
\roots_0 = \{(\la, \mu) \in \mathfrak{a}_+\times \mathfrak{a} : \mu \in \roots_\la\}.
$$ We claim that $\roots_0$ is convex. Let $(\la_1, \mu_1), (\la_2, \mu_2) \in \roots_0,$
$$
t(\la_1, \mu_1) + (1-t)(\la_2, \mu_2) = (t\la_1 + (1-t)\la_2, t\mu_1 + (1-t)\mu_2).
$$
Now, from \cite{gichev-polar} we have
$$
\roots_{(t\la_1 + (1-t)\la_2)} = \roots_{t\la_1} + \roots_{(1-t)\la_2} = t\roots_{\la_1} + (1-t)\roots_{\la_2},
$$
and so
$$
t\mu_1 + (1-t)\mu_2\in \roots_{(t\la_1 + (1-t)\la_2)}.
$$
This shows that $\roots_0$ is convex.
Let
$$
\pi_1 : \mathfrak{a}_+\times \mathfrak{a} \to \mathfrak{a}_+
$$ and
$$
\pi_2 : \mathfrak{a}_+\times \mathfrak{a} \to \mathfrak{a}.
$$
Then,
$$
S^{\#} = \pi_2(\pi_1^{-1}(S)\cap \roots_0)
$$
and so it is convex.
\end{proof}
\begin{theorem}
Let $\mu_\mathfrak{p} : X \to \mathfrak{p}$ be the gradient map. Let $\mathfrak{a}\subset \mathfrak{p}$ be a maximal Abelian subalgebra and let $\mu_\mathfrak{a} = \pi_\mathfrak{a} \circ \mu_\mathfrak{p}$ be the corresponding gradient map. If $\mu_{\mathfrak p} (X) \cap \mathfrak a_+$ is convex, then
\begin{equation}\label{Convex2}
\mu_\mathfrak{a}(X) = (\mu_\mathfrak{p}(X) \cap \mathfrak{a}_+)^\# ,
\end{equation}
and so convex.
\end{theorem}
\begin{proof}
The action of $K$ on $\mathfrak{p}$ is polar and $\mathfrak{a}$ is a section \cite{Da}. Moreover, if $x\in \mathfrak{p},$ then
$$
K\cdot x \cap \mathfrak{a}_+ = \{\la\}.
$$
By a beautiful Theorem of Kostant \cite{kostant-convexity}, $\pi_\mathfrak{a} (K\cdot x )=\Delta_\lambda$ and so a polytope.
Therefore
$$
\mu_\mathfrak{a}(X) = \{\mu\in \mathfrak{a} : \mu \in \roots_\la, \quad \text{where}\quad \la \in \mu_\mathfrak{p}(X)\cap \mathfrak{a}_+\} = (\mu_\mathfrak{p}(X) \cap \mathfrak{a}_+)^\#.
$$
By the above Theorem, $\mu_\mathfrak{a}(X)$ is convex.
\end{proof}
\subsection{Convexity Results of the Gradient Map}
In this section, we continue to investigate the Abelian convexity property of the gradient map. We give a new proof of the convexity property of the gradient map for $X = Z$ avoiding the Linearization theorem.
\begin{theorem}\label{Kahlar}
Suppose $(Z, \omega)$ is a connected and compact \Keler manifold. Then
$$\mu_\mathfrak{a} : Z \to \mathfrak{a}$$ is a convex polytope.
\end{theorem}

\begin{proof}
Let $\xi\in \mathfrak{a}$ and $\{t_n\}_{n\in \mathbb{N}}$ be a sequence such that $t_n \to \infty.$ Denote by $g_n = \text{exp}(t_n\xi).$
By Theorem 2 in \cite{LG2}, up to passing to a subsequence, there exist a proper analytic subset $U$ of $Z$ such that
$$
g_n : Z\setminus U \to Z, \quad g_n \to \phi_\infty,
$$ 
where $\phi_\infty$ is Non-dominant meromorphic map. Note that $Z\setminus U$ is connected and Zarinski open. Since $g_n = \text{exp}(t_n\xi),$ it follows $\phi_\infty (Z\setminus U) \subset Z^{\xi}$ where
$$
Z^{\xi} = \{x\in Z : \xi_M (x) = 0\}.
$$
The vector $J\xi$ is a Killing vector field and so $Z^{\xi^\#}$ is smooth, possibly disconnected, submanifold of $Z$ \cite{kot}. Since $Z\setminus U$ s connected, $\phi_\infty(Z\setminus U)$ is contained in a connected components of $Z^{\xi^\#}$ which we denoted by $Z_0.$

Set $\mu_\mathfrak{a}^\xi := \langle \mu_\mathfrak{a}, \xi \rangle.$ $\mu_\mathfrak{a}^\xi(Z_0)$ is constant. We claim that
$$
\mu_\mathfrak{a}^\xi(Z_0) = \text{max}_{z\in Z}\mu_\mathfrak{a}^\xi.
$$
Indeed, let $x_0\in Z:$
$$
\mu_\mathfrak{a}^\xi(x_0) = c = \text{max}_{z\in Z}\mu_\mathfrak{a}^\xi.
$$ Let $\epsilon > 0.$ There exist neighbourhood $N$ of $x_0$ such that
$$
\mu_\mathfrak{a}^\xi(N) \subset (c-\epsilon, c].
$$ $(Z\setminus U)\cap N \neq \emptyset.$ Pick $p\in (Z\setminus U)\cap N.$ Then
$$
c-\epsilon \leq \mu_\mathfrak{a}^\xi(p) \leq \mu_\mathfrak{a}^\xi(\phi_\infty (p)).
$$
Therefore
$$
\mu_\mathfrak{a}^\xi(Z_0) = c = \text{max}_{z\in Z}\mu_\mathfrak{a}^\xi.
$$

Let $P = \text{Conv}(\mu_\mathfrak{a}(Z))$. By a Theorem of Atiyah \cite{Atiyah}, see also  \cite{LG,heinzner-stoetzel}, for any $p\in Z$ we have
\begin{equation}\label{ui}
\mu_\mathfrak{a}(\overline{A\cdot p}) = \text{Conv}(\mu_\mathfrak{a}(Z^A \cap \overline{A\cdot p})) \subset \mu_{\mathfrak a} (p) + \lia_p,
\end{equation} 
where
$Z^A = \{z\in Z : A\cdot z = z\}$ and $\mathfrak a_p$ is the Lie algebra of $A_p$. Since $Z$ is compact, the set $Z^A$
has finitely many connected components. By $(\ref{ui})$, it follows that $\mu_\mathfrak{a}(Z^A)$ is finite and $P = \text{Conv}(\mu_\mathfrak{a}(Z^A))$. This implies that $P$ is a polytope.

Let $x_0, x_1, \cdots, x_n \in P$ be verticies. Since $P$ is a polytope any face is exposed. Then there exist $\xi_0, \xi_1, \cdots, \xi_n \in \mathfrak{a}$ such that
$$
x_i = \{\theta \in P: \langle\theta, \xi_i\rangle = \text{max}_{y\in P}\langle y, \xi_i\rangle, i = 0, 1, \cdots, n  \}.
$$
Denote $c_i = \langle x_i, \xi_i\rangle$, for $i=1,\ldots,n$. There exists $U_0, \cdots, U_n$ proper analytic subset and a sequence $t_N \to +\infty$ such that
$
\lim_{N \to \infty}\text{exp}(t_N\xi_i)
$ exists in $Z\setminus U_i.$ Moreover, if $z_i\in Z\setminus U_i,$ then

$$
\lim_{N \to \infty}\text{exp}(t_N\xi_i)\cdot z_i \in (\mu_\mathfrak{a}^{\xi_i})^{-1}(c_i) = \mu_\mathfrak{a}^{-1}(x_i).
$$
Now,

$$
(Z\setminus U_0) \cap (Z\setminus U_1) \cap \cdots \cap (Z\setminus U_n) = Z\setminus (U_0 \cup  U_1 \cup \cdots \cup  U_n)\neq \emptyset,
$$ then
$$
\overline{A\cdot p}\cap \mu_\mathfrak{a}^{-1}(x_i) \neq \emptyset,
$$ whenever $p\in Z\setminus (U_0 \cup  U_1 \cup \cdots \cup  U_n)$ for any $i = 0, \cdots, n.$  Applying, again  a Theorem of Atiyah \cite{Atiyah}, we have
$$
\mu_\mathfrak{a}(\overline{A\cdot p}) = \text{Conv}(\mu_\mathfrak{a}(Z^A)\cap \overline{A\cdot p}) = P.
$$ Therefore
$$
\mu_\mathfrak{a}(Z) = P.
$$
\end{proof}
\begin{corollary}
In the above setting, the following hold true:
\begin{enumerate}
\item $\{p\in Z: \mu_\mathfrak{a}(\overline{A\cdot p}) = \mu_\mathfrak{a}(Z) \}$ contains an open and dense subset of $Z.$
\item Any local maximal of $\mu_\mathfrak{a}^\xi$ is a global maximal. Indeed, we have proved that the unstable manifold of the critical component $C_0$ corresponding to the maximum is Zarinski open.
\end{enumerate}
\end{corollary}

We now prove the convexity property of the gradient map when $X$ is a connected, compact coisotropic submanifold of $(Z,\omega).$
\begin{definition}\label{coisotropic}
A submanifold $X\subset (Z, \omega)$ is coisotropic if for any $p\in X,$ we have
$$
(T_pX)^{\bot_\omega} \subset T_pX.
$$
\end{definition}

Since $(Z, \omega)$ is \keler,

$$
(T_pX)^{\bot_\omega} = J((T_pX)^{\bot}).
$$

\begin{lemma}\label{lemaa}
If $X$ is coisotropic, then for any $p\in X$, we have
$$
T_pX + J(T_pX) = T_pZ.
$$
\end{lemma}
\begin{proof}
$$
J((T_pX)^{\bot}) \subset T_pX.
$$
Applying $J$ we have
$$
(T_pX)^{\bot} \subset J(T_pX).
$$
And so
$$
T_pX + J(T_pX) = T_pZ.
$$
\end{proof}
\begin{lemma}\label{open} Let $X$ be a $A$-invariant compact connected coistropic submanifold of $(Z,\omega).$
Let $\xi\in \mathfrak{a}.$ Then
$$
\text{max}_{p\in X}\mu_\mathfrak{a}^\xi = \text{max}_{z\in Z}\mu_\mathfrak{a}^\xi.
$$ Moreover, the unstable manifold associated to the maximum of $\mu_\mathfrak{a}^\xi$ is open and dense.
\end{lemma}
\begin{proof}
Let $W_0^\xi$ be the unstable manifold of the critical manifold $C_0$ satisfying $\mu_\mathfrak{a}^\xi(C_0) = c_0.$ Assume that $C_0$ corresponds to a local maximum. Since
$$
\nabla \mu_\mathfrak{a}^\xi|_X = \nabla \mu_\mathfrak{a}^\xi,
$$ it follows that

$$
W_0^\xi = \bar{W_0^\xi} \cap X,
$$ where $\bar{W_0^\xi}$ is the unstable manifold in $Z$ of the critical components $\bar{C}_0$ such that
$$
\mu_\mathfrak{a}^\xi(\bar{C}_0) = \mu_\mathfrak{a}^\xi(C_0) = c_0.
$$
By a Linearization theorem in \cite{PG}, $\bar{W_0^\xi}$ is a complex manifold and $W_0^\xi$ is open in $X$. Let $p\in W_0^\xi.$ Since
$$
T_pW_0^\xi = T_pX \subset T_p\bar{W_0^\xi},
$$
it follows that
$$
T_pX + J(T_pX) \subset T_p\bar{W_0^\xi}.
$$ By Lemma \ref{lemaa}, $\bar{W_0^\xi}$ is open. Since $\mu_\mathfrak{a}^\xi :Z \lra \R$ is Morse-Bott of even index, it follows that $\mu_\mathfrak{a}^\xi:Z \lra \R$ has a unique local maximum and so $\bar{W_0^\xi}$ is open and dense. Therefore $\mu_\mathfrak{a}:X \lra \R$ has also a unique local maximum. This proves 
$$\text{max}_{p\in X}\mu_\mathfrak{a}^\xi = \text{max}_{z\in Z}\mu_\mathfrak{a}^\xi.$$
 Since $\mu_\mathfrak{a}^\xi:X \lra \R$ is Morse-Bott,  applying Theorem \ref{decomposition} we get that, the unstable manifolds different from $W_0^\xi$ have codimension at least one. This implies that $W_0^\xi$ is also open and dense in $X$.
\end{proof}
\begin{theorem}\label{convexity-coisotropic}
If $X$ is a $A$-invariant compact connected coistropic submanifold of $(Z,\omega).$ Then
$$
\mu_\mathfrak{a}(X) = \mu_\mathfrak{a}(Z),
$$ and so a polytope. Moreover, there exists a subset an open and dense subset $W$ of $X$  such that for any $p\in W$, we have
\[
\mu_\mathfrak{a}(X) =\overline{\mu_{\mathfrak a} (A\cdot p)}.
\]
\end{theorem}
\begin{proof}
Let $\xi\in \mathfrak{a},$ by Lemma \ref{open},
$$
\text{max}_{p\in X}\mu_\mathfrak{a}^\xi = \text{max}_{z\in Z}\mu_\mathfrak{a}^\xi,
$$ and the unstable manifold associated to the maximum of $\mu_\mathfrak{a}^\xi$ is open and dense. By Proposition 3.1 in \cite{LG}, $\mu_\lia (X)$ is a polytope. Moreover,  the set
\[
\left\{p\in X : \mu_\mathfrak{a}(\overline{A\cdot p}) = \mu_\mathfrak{a}(X)\right\}
\]
is open and dense. Finally, by Proposition \ref{convex-criterium}, we have
$$
\mu_\mathfrak{a}(X) = \mu_\mathfrak{a}(Z)
$$
concluding the proof.
\end{proof}
\section{Two orbits variety}
In this section we investigate two orbits variety.
\begin{definition}
Let $X$ be a compact and connected $G$-stable submanifold of $(Z,\omega)$. We say that $X$ is a two orbit variety if $G$-action on $X$ has two orbits.
\end{definition}
S. Cupit-Foutou obtained the classification of a complex algebraic varieties on which a reductive complex algebraic group acts with two orbits \cite{Cupit-Foutou}.

The norm square $f$ has a maximum and a minimum. By the stratification theorem, keeping in mind that the strata are $G$-invariant, $X$ is the union of a closed $G$-orbit $S_{\beta_{max}}$, where the norm square achieves the maximum, and an open $G$-orbit $S_{\beta_{min}}$, the stratum relative to the minimum of the norm square. We then show that $f$ is a Morse-Bott function.
\begin{theorem}\label{two-orbit-Morse} If $G$ acts on $X$ with two orbits, then
\begin{enumerate}
    \item the function $f : X \rightarrow \mathbb{R}$ given by
\begin{equation*}
    f(x) := \frac{1}{2}\parallel \mu_\mathfrak{p}(x)\parallel^2 \qquad \text{for}\quad x \in X.
\end{equation*} is Morse-Bott; It has only two connected critical submanifolds given by the closed $G$-orbit $S_{\beta_{max}}$, the stratum associated with the maximum of $f$ and by a $K$-orbit $S_{\beta_{min}}$, the stratum associated with the minimum of $f$.

    \item The Poincar\'{e} polynomial $P_X(t)$ of $X$ satisfies
    $$
    P_X(t) = t^k\cdot P_{S_{\beta_{max}}}(t) +P_{S_{\beta_{min}}}(t) - (1+t)R(t),
    $$
    where $k$ is the real codimension of $S_{\beta_{max}}$ in $X$ and $R(t)$ is a polynomial with positive integer coefficients. In particular $\chi(X) = \chi(S_{\beta_{max}}) + \chi(S_{\beta_{min}});$

    \item The $K$-equivariant Poincar\'{e} series of $X$ is given by

    $$
    P^K_X(t) = t^k\cdot P^K_{S_{\beta_{max}}}(t) + P^K_{S_{\beta_{min}}}(t).
    $$
\end{enumerate}
\end{theorem}
\begin{proof}
We first proof (a). Consider the function $f$ and its critical set $C.$ $f$ is non constant on $X;$ in fact if $f$ is constant, then every point of $X$ is a maximum point and in view of proposition \ref{lemm}, all $G$-orbit would be closed.


By Theorem \ref{critt}, we have that $S_{\beta_{min}}$ consist of a single $K$-orbit, and so it is connected.

Since $f$ realizes its maximum value at any critical point $x$ belonging to $S_{\beta_{max}},$ then by Proposition \ref{Hessian comp}d
$$
H_x(f) < 0\quad \text{on} \quad T_x(S_{\beta_{max}})^\bot.
$$
Now we show that the Hessian of $f$ at a critical point $x$ belonging to $S_{\beta_{min}}$ is non degenerate in the normal direction. Set $\mu_\mathfrak{p} (x) = \beta_{min} = \beta.$

Suppose $\beta \neq 0.$ By Remark \ref{Hessian comp1},

$$T_xS_{\beta_{min}} = T_x(G\cdot x) = T_x(K\cdot x) \oplus \mathfrak{p}^\beta \cdot x \oplus \lier^{\beta+}\cdot x,
$$
where $\lier^{\beta+}$ is the Lie algebra of $R^{\beta+}.$ By Proposition \ref{Hessian comp}b,

$$
H_x(f) > 0\quad \text{on} \quad \mathfrak{p}^\beta \cdot x \oplus \lier^{\beta+}\cdot x.
$$ Since $H_x(f)\geq 0,$ it follows that
$$
T_x(G\cdot x) = T_x(K\cdot x) \oplus^\bot (\mathfrak{p}^\beta \cdot x \oplus \lier^{\beta+}\cdot x).
$$

Suppose $\beta = 0$. Let $x = x_{min}.$ By Theorem \ref{critt}, $\mu_\mathfrak{p}^{-1}(0) = K\cdot x.$  
$$
\text{ker}\;\;d\mu_\mathfrak{p}(x) = (\mathfrak{p} \cdot x)^\bot.
$$
By Proposition \ref{Hessian}, $$H_x(f)|_{(\mathfrak{p}\cdot x)} > 0.
$$
$$
T_{x}X = T_{x}(G\cdot x).
$$



Since $K\cdot x_{min} \subset \text{ker}\;\;d\mu_\mathfrak{p}(x_{min}),$
it follows that

$$
T_{x}(G\cdot x) = K\cdot x + \mathfrak{p}\cdot x = T_{x}X
$$
$$
H_x(f)|_{(K\cdot x)} = 0.
$$
$$
T_{x}X = K\cdot x \oplus^\bot \mathfrak{p}\cdot x.
$$ By dimensional reason, $H_x(f)$ is non degenerate.



These show that $H_x(f)$ is non degenerate. Hence $f$ is Morse-Bott. The statements in (b) and (c) follow from the general theory in \cite{Kirwan}.
\end{proof}
Finally, we point out that the Abelian convexity Theorem holds for a two orbit variety. Indeed, $X$ has a unique closed orbit. By  Theorem \ref{convexity property} we derive the following result.
\begin{theorem}
Let $X$ be a two orbits variety. Let $\lia \subset \liep$ be an Abelian subalgebra. Then $\mu_\mathfrak{a}(X)$ is a polytope.
\end{theorem}

\end{document}